\documentclass{siamltex}
\usepackage{amsmath}
\usepackage{amsfonts}
\usepackage{amssymb}
\usepackage{graphicx}
\usepackage{color}
\usepackage{subfigure}
\usepackage{cite}
\usepackage{comment}

\allowdisplaybreaks
\usepackage[left=1.5in, top=1in, right=1.5in, bottom=1in]{geometry}

\newtheorem{algorithm}[theorem]{Algorithm}

\newtheorem{remark}[theorem]{Remark}

\newcommand{\be}{\begin{equation}}
\newcommand{\ee}{\end{equation}}
\newcommand{\bea}{\begin{eqnarray}}
\newcommand{\eea}{\end{eqnarray}}
\newcommand{\beas}{\begin{eqnarray*}}
\newcommand{\eeas}{\end{eqnarray*}}

\newcommand{\vertiii}[1]{{\left\vert\kern-0.25ex\left\vert\kern-0.25ex\left\vert #1 
    \right\vert\kern-0.25ex\right\vert\kern-0.25ex\right\vert}}

\input{tcilatex}
\begin{document}

\title{Analysis and approximation of a fractional Laplacian-based closure model for turbulent  flows and its connection to Richardson pair dispersion\thanks{Supported by the US National Science Foundation grant DMS-1315259 and the US Air Force Office of Scientific Research grant FA9550-15-1-0001.}}
\author{Max Gunzburger\thanks{Department of Scientific Computing, Florida State University, Tallahassee FL 32306-4120; {\tt mgunzburger@fsu.edu}.}
\and 
Nan Jiang\thanks{
Department of Scientific Computing, Florida State University, Tallahassee, FL 32306-4120. Current address: Department of Mathematics and Statistics, Missouri University of Science and Technology, Rolla MI 65409-0020; {\tt jiangn@mst.edu.}} \and 
Feifei Xu\thanks{
Department of Scientific Computing, Florida State University, Tallahassee, FL 32306-4120. Current address: Department of Mathematics, University of North Carolina, Chapel Hill NC 27599-3250; {\tt winterflyfei@gmail.com}.}
}
\date{\today}
\maketitle

\begin{abstract}
We study a turbulence closure model in which the fractional Laplacian $(-\Delta)^\alpha$ of the velocity field represents the turbulence diffusivity. We investigate the energy spectrum of the model by applying Pao's energy transfer theory. For the case $\alpha=1/3$, the
corresponding power law of the energy spectrum in the inertial range has a correction
exponent on the regular Kolmogorov -5/3 scaling exponent. For this case, this model represents  Richardson's particle pair-distance superdiffusion of a fully developed homogeneous turbulent flow as well as L\'evy jumps that lead to the superdiffusion. For other values of $\alpha$,  the power law of the energy spectrum is consistent with the regular Kolmogorov -5/3 scaling exponent. We also propose and study a modular time-stepping algorithm in semi-discretized form. The algorithm is minimally intrusive to a given legacy code for solving Navier-Stokes equations by decoupling the local part and nonlocal part of the equations for the unknowns. We prove the algorithm is unconditionally stable and unconditionally, first-order convergent. We also derive error estimates for full discretizations of the model which, in addition to the time stepping algorithm, involves a finite element spatial discretization and a domain truncation approximation to the range of the fractional Laplacian. 
\end{abstract}


%%%%%%%%%%%%%%%%%%%%%%%%%
\section{Introduction}

Nonlocal models have attracted intensive research interests in recent years due to their ability to model phenomena that cannot be correctly described by classical partial differential equation models. Many advances have been made in various scientific and engineering areas including continuum mechanics \cite{Silling00}, graph theory \cite{LS06}, image denoising \cite{BCM10}, machine learning \cite{RBV10}, and phase transitions \cite{BC99}. In particular, fractional derivative models have been found to be effective in modeling anomalous diffusion processes \cite{MK04,MT04}. In this work, we study a new closure model based on the fractional Laplacian operator that accounts for the anomalous diffusion (superdiffusion in this case) that arises in fully-developed turbulent fluid flows.

Turbulence modeling remains one of the most challenging scientific problems. Despite the fact that the governing equations for turbulence have been known since 1845, a full understanding of turbulence is still far from complete due to its extremely complex behavior and chaotic nature.  The wide range of scales present in turbulence results in a very high computational complexity and renders direct numerical simulations infeasible even with modern supercomputers. Thus, turbulence models are introduced to predict the mean flow and coherent structures with the effects of the turbulence on the mean flow being modeled. The mean flow and the smaller scales of turbulence interact through a quantity referred to as the Reynolds stress that appears in the evolution equations of the mean flow and which, to close the system, must be replaced by terms that are solely dependent on the mean flow; this is the closure problems of turbulence. It is worth noting that all turbulence models inevitably invoke additional heuristic hypotheses and thus tend to work only for a narrow class of problems. In this paper, we consider a class of nonlocal operators, namely fractional Laplacian operators, as a turbulent closure model \cite{Chen06}. These operators have a deep connection with L\'evy jump processes in probability theory and corresponding superdiffusion behavior in turbulent flows. 

The closure model we consider is given by
\begin{equation}\label{model}
\left\{\begin{aligned}
u_t + (u\cdot\nabla) u -\nu \Delta u+\gamma (-\Delta )^{\alpha} u+\nabla p&= f\qquad \text{in} \;\; (0,T] \times \Omega\\
 \nabla \cdot u&=0 \qquad \text{in} \;\; (0,T] \times \Omega
\end{aligned}\right.
\end{equation}
for $\alpha\in(0,1)$, where $\Omega\subset R^d$ denotes a bounded, open domain and $[0,T]$ a temporal interval of interest. The fractional Laplacian operator is most often defined in terms of Fourier transforms as
$$
(-\Delta )^{\alpha} u(x) = \mathcal{F}^{-1}\big(|\xi|^{2\alpha}\mathcal{F}(u)(\xi)\big)(x)
\qquad\mbox{for $x\in R^d$},
$$
where $\mathcal{F}$ denotes the Fourier transform and $\mathcal{F}^{-1}$ its inverse. For $\alpha\in(0,1)$, an equivalent definition \cite{Applebaum04} is
\begin{equation}\label{ffll}
(-\Delta)^\alpha u = C_{d,\alpha} \int_{R^d}\frac{u(x)-u(y)}{|x-y|^{d+2\alpha}}dy,
\end{equation}
where $C_{d,\alpha}=\alpha2^{2\alpha}\frac{\Gamma(\frac{d+2}{2})}{\Gamma(\frac{1}{2})\Gamma(1-\alpha)}$ is a normalizing constant.

We first explore, in \S\ref{energytransfer}, the energy spectrum of the model \eqref{model} by applying Pao's transfer theory \cite{Pao65}. Following the methodology developed in \cite{NLW11} for the analysis of a family of approximate deconvolution models of turbulence, we derive an expression for the long-time averaged energy distribution. The results show that if $\alpha=1/3$, the corresponding power law of the energy spectrum in the inertial range has a correction to the well-known Kolmogorov $-5/3$ scaling exponent. For this case, this model \eqref{model} corresponds to the Richardson particle pair-distance superdiffusion of a fully developed homogeneous turbulent flow (see \cite{Rich26} and also \cite{LALV04,JPT99,GLMMR04}) and also to a L\'evy jump process that lead to superdiffusion. For other values of $\alpha$ in $(0, 1)$, the power law of the energy spectrum is consistent with the standard Kolmogorov $-5/3$ scaling exponent. 

The use of the fractional Laplacian operator results in a dense matrix that requires different types of linear solvers than the ones employed for the sparse matrices that arise when solving discretized local turbulent models. Because many industrial flow codes are highly optimized and extensively calibrated, a direct implementation of the proposed model encounters many technical issues and requires substantial code modifications. To ease the implementation process for the model we consider and reduce the required effort, we propose, in \S\ref{imexmethod}, a novel modular algorithm that splits the local and nonlocal parts of the equations so that only minimal changes need be done to legacy codes. The algorithm consists of two steps. The first is to solve the Navier-Stokes equations with a modified right-hand side so that a legacy code can be easily modified without changing linear solvers used or the manner in which matrices are stored. The second step is a post-processing step that involves solving a discretized  fractional Laplacian problem. This step can be added to the legacy code as a separate routine. We study our modular algorithm based on a first-order time-stepping method also given in \S\ref{imexmethod}, proving that the algorithm is unconditionally stable and unconditionally first-order convergent. This modular algorithm can be extended to higher-order time-stepping methods. 

The model we study employs a standard definition, i.e., \eqref{ffll}, of the fractional Laplacian operator, instead of some variants that are defined on the bounded domain $\Omega$ for which, e.g., the integral appearing in \eqref{ffll} is replaced by an integral over $\Omega$. Thus, in our model, although the domain of the fractional Laplacian operator appearing in \eqref{model} is the bounded domain $\Omega$, its range is the infinite domain $R^d$. In practice, however, the integral in \eqref{ffll} is approximated by an integral over a finite domain strictly {\em containing} the given domain $\Omega$. In \S\ref{truncatedinteractions}, for a fixed bounded domain $\Omega$, we derive an estimate for the error incurred by truncation as the extent of the truncated containment domain increases, in particular showing that solutions of the truncated domain problem converge to those of \eqref{model}.

In \S\ref{sec:femfem}, we complete our study by defining and analyzing full discretizations of the problem \eqref{model} for which finite element spatial discretizations are added to the time-stepping methods of \S\ref{imexmethod} and the domain truncation of \S\ref{truncatedinteractions}.  

%%%%%%%%%%%%%%%%%%%%%%%%
\section{Energy transfer}\label{energytransfer}
We investigate the energy spectrum of the new model based on the energy transfer theory of Pao \cite{Pao65,Pao68}. We thus consider the Navier-Stokes equations in a periodic box $\Omega=(0,2\pi)^3$ in $R^3$:
\begin{equation}\label{periodicNSE}
\left\{\begin{aligned}
u_t + (u\cdot\nabla) u -\nu \Delta u+\gamma (-\Delta )^{\alpha} u+\nabla p&= f \,\,\,\qquad\text{for $x\in (0, 2\pi)^3$ and $t>0$}\\
 \nabla \cdot u&=0 \qquad \,\,\,\text{for $x\in (0, 2\pi)^3$ and $t>0$} \\
 u&=u_0 \qquad\text{for $x\in (0, 2\pi)^3$ and $t=0$,}
\end{aligned}\right.
\end{equation}
where, for $\phi=u$, $p$, $u_0$, or $f$,
$$
\phi(x+2\pi e_j, t)= \phi(x, t)\qquad\text{and}\qquad \int_{\Omega} \phi dx =0  
$$
with $e_j$, $j=1,2,3$, denoting the Cartesian unit vectors. Then, the fluid velocity $u(x,t)$  can be expanded in Fourier series as 
$$
u(x,t)=\sum_{\bold{k}} {u}(\bold{k},t)e^{i\bold{k}\cdot x} 
$$
and its associated kinetic energy is given by
$$
E(t)=\sum_{\bold{k}}\frac{1}{2}|{u}(\bold{k}, t)|^2,
$$
where $\bold{k}=(k_1, k_2, k_3)$ with $k_j$$, j=1, 2,3$, being non-negative integers and $\bold{k}\neq (0,0,0)$. We partition the kinetic energy into wave number shells given by
$$
E(k,t)= \sum_{k=\vert \bold{k}\vert}\frac{1}{2}\vert{u}(\bold{k},t)\vert^2,
$$
where $\vert \bold{k}\vert^2=k_1^2+k_2^2+k_3^2$, so that the total energy is then given by
$$
E(t)= \sum_{1\leq k}E(k,t).
$$
An evolution equation for $E(k,t)$ can be derived by taking inner product of  \eqref{periodicNSE} with a single Fourier mode and then summing over all modes; see Davidson \cite{D04} or Pope \cite{Pope}. Using the Kronecker delta, $E(k,t)$ satisfies
$$
\begin{aligned}
\frac{\partial}{\partial t} E(k,t) + \sum_{\vert \bold{j}\vert=k}\sum_{\bold{k_1}}\sum_{\bold{k}_2}&
\Big\{{u}(\bold{k}_1, t)\cdot {u}(\bold{k_2},t) \otimes \bold{k}_2\cdot \overline{{u}(\bold{j},t)} \delta_{\bold{k}_1+\bold{k}_2, \bold{j}}\Big\}\\&
+2\nu k^2E(k,t)+ 2\alpha \gamma k^{2\alpha}E(k,t) = \sum_{\vert \bold{j}\vert=k}{f}(\bold{j}, t)\cdot \overline{{u}(\bold{j},t)},
\end{aligned}
$$
where $\overline{u}$ denotes the complex conjugate of $u$. Note that for each $\bold{k}$, $ u(\bold{k}, t) $ is a complex vector that satisfies conjugate symmetry, i.e. $u(\bold{k}, t)= \overline{u(\bold{-k}, t)}$.
Define the energy transfer function $S(k,t)$ by
$$
S(k,t) : = \sum_{1\leq k'\leq k}T(k', t),
$$
where
$$
T(k,t) =\sum_{\vert \bold{j}\vert=k}\sum_{\bold{k_1}}\sum_{\bold{k}_2}
\Big\{{u}(\bold{k}_1, t)\cdot {u}(\bold{k_2},t) \otimes \bold{k}_2\cdot \overline{{u}(\bold{j},t)} \delta_{\bold{k}_1+\bold{k}_2, \bold{j}}\Big\}.
$$
Because fully developed, homogeneous, isotropic turbulence is characterized by a wide range of persistent scales, in transfer theory $k$ is considered as a continuous variable. We then redefine the energy transfer function as
$$
S(k,t)= \int_{0}^k T(k', t)dk' 
\qquad\mbox{so that}\qquad
T(k,t)= \frac{\partial}{\partial k} S(k, t).
$$
We further assume that for all $t>0$ the energy is input into the $k=1$ modes by smooth, persistent body forces, i.e., $E(1,t)=\frac{1}{2}U^2$, where $U=U(t)$ is fixed, which is a representative large scale velocity. Based on all the assumptions made above, we have the evolution equation for the kinetic energy for a given $k$ given by
$$
\left\{\begin{aligned}
&\frac{\partial}{\partial t} E(k,t) + \frac{\partial}{\partial k}S(k,t)+2\nu k^2E(k,t)+ 2\alpha\gamma k^{2\alpha}E(k,t) = 0\quad \text{for } 1<k<\infty,\, t>0
\\
&E(1,t)=\frac{1}{2}U^2\quad\text{for } t>0
\\
&E(k,0)=E_0(k)\quad \text{ for } 1<k<\infty \quad\text{and}\quad E_0(k)=0 \,\,\,\text{for large } k.
\end{aligned}
\right.
$$

This system is not closed. So far the most successful closure model is Pao's model \cite{Pao68} given by
$$
S(k,t)=C_k^{-1}\epsilon_0^{1/3}k^{5/3}E(k,t),%\quad \text{where}\quad \epsilon_0=2^{-3/2}C_k^{-1}U^3
$$
where $\epsilon_0=2^{-3/2}C_k^{-1}U^3$ and $C_k$ is the Kolmogorov constant. Based on Pao's closure model, we now consider the problem
$$
\left\{\begin{aligned}
&\frac{\partial}{\partial t} E(k,t) + \frac{\partial}{\partial k}(C_k^{-1}\epsilon_0^{1/3}k^{5/3}E(k,t))+2\nu k^2E(k,t)\\
&\qquad\qquad\qquad\qquad
+ 2\alpha \gamma k^{2\alpha}E(k,t) = 0\quad \text{for } 1<k<\infty,\, t>0\\
&E(1,t)=\frac{1}{2}U^2\quad \text{ for } t>0\\
&E(k,0)=E_0(k)\quad \text{ for } 1<k<\infty \quad\text{and}\quad E_0(k)=0 \,\,\,\text{for large } k.
\end{aligned}
\right.
$$

We define the long-time averaged energy distribution as
$$
E(k)=\lim_{T\rightarrow\infty}\frac{1}{T}\int_0^T E(k,t) dt
$$
which then satisfies
$$
 \frac{\partial}{\partial k}\big(C_k^{-1}\epsilon_0^{1/3}k^{5/3}E(k)\big)+(2\nu k^2+2\alpha \gamma k^{2\alpha})E(k) = 0\quad \text{for } 1<k<\infty
$$
along with $E(1)=\frac{1}{2}U^2$. This equation can be easily solved and the solution is given by
\begin{equation}\label{Spectrum}
E(k) 
\left\{
\begin{array}{ll}
\frac{1}{2}U^2e^{\beta_2}k^{-(\frac{5}{3}+\beta_1)}exp(-\beta_2k^\frac{4}{3})& \text{if }\alpha=\frac{1}{3} \\[1ex] 
\frac{1}{2}U^2e^{\beta_3}k^{-\frac{5}{3}}exp(-\beta_2k^\frac{4}{3})exp(-\beta_4 k^{(2\alpha-\frac{2}{3})}) &\text{if } \alpha\in(0,1) \text{ with } \alpha\neq \frac{1}{3},  
\end{array}\right.
\end{equation}
where
$$
\beta_1=\frac{2C_k\gamma}{3\epsilon_0^{1/3}},\quad\,\,  \beta_2 =\frac{3C_k\nu}{2\epsilon_0^{1/3}},\quad\,\, \beta_3 =\frac{3C_k\nu}{2\epsilon_0^{1/3}}+\frac{2\alpha C_k\gamma}{(2\alpha-\frac{2}{3})\epsilon_0^{1/3}},\quad\,\, \beta_4=\frac{2\alpha C_k\gamma}{(2\alpha-\frac{2}{3})\epsilon_0^{1/3}}.
$$

%%%%%%%%%%%%%%%
\subsection{The inertial range energy spectrum} 

In the inertial range, the viscous dissipation effect is negligible because $\nu$ is small. Then, over this range, the expression \eqref{Spectrum} for $E(k)$ reduces to
$$
E(k) =
\begin{cases}
\frac{1}{2}U^2e^{\beta_2}k^{-(\frac{5}{3}+\beta_1)}& \text{if }\alpha=\frac{1}{3} \\ 
\frac{1}{2}U^2e^{\beta_3}k^{-\frac{5}{3}}&\text{if } \alpha\in(0,1) \text{ with } \alpha\neq \frac{1}{3}.
\end{cases}
$$
This shows that if the exponent $\alpha$ of the fractional Laplacian is equal to $1/3$, the corresponding power law of the energy spectrum in the inertial range has a deviation from the exponent of the regular Kolmogorov $-5/3$ scaling exponent, whereas for other values of $\alpha\in(0,1)$ the power law of the energy spectrum is consistent with the Kolmogorov theory. The Kolmogorov scaling theory (often referred to as the ``K41 theory'') is the most celebrated turbulence theory and is supported by much experimental evidence from atmospheric and oceanographic turbulence at sufficiently high Reynolds number \cite{F95,D04}. {\color{black}However, small deviations from the $-5/3$ scaling exponent have also been observed in various turbulence experiments \cite{AGHA84, KIYIU03, GFN02, F95}. These deviations, although small in the spectrum, considerably affect higher-order statistics.} There have been many theoretical attempts to modify the exponent in the power law. Actually, Kolmogorov himself first proposed a modification of the exponent \cite{Kolmogorov62}. Most of the theories developed concern the intermittency in the inertial range and various intermittency models have been built to try to fit experimental data such as the $\beta$-model; see \cite{F95} for a review. However, as far as we know, there exists no partial differential equation turbulence model that is able to correct the exponent of the Kolmogorov $-5/3$ scaling exponent as does the model we consider.

The fractional Laplacian is the generator of $\alpha$-stable L\'evy processes in probability theory. The special case of $\alpha=\frac{1}{3}$ that leads to a correction exponent in the power law of energy spectrum corresponds to the $\frac{2}{3}$-stable L\'evy process. This has an interesting connection with the Richardson's particle pair-distance superdiffusion in a fully developed homogeneous turbulence for which $\left\langle r^2\right\rangle = \overline{C}\epsilon\Delta t^3$ so that the displacement increment also obeys the $\frac{2}{3}$-stable L\'evy distribution. Thus, the fractional Laplacian with $\alpha=\frac{1}{3}$ which we use in our model actually introduces the corresponding L\'evy flight mechanism into the system and represents Richardson's turbulence superdiffusion.

%%%%%%%%%%%%%%%%%
\subsection{The dissipation range energy spectrum} 

In the dissipation range, viscous dissipation is dominant and removes energy from the system. We rewrite \eqref{Spectrum} as
$$
E(k) =\begin{cases}
\frac{1}{2}U^2e^{\beta_2}k^{-(\frac{5}{3}+\beta_1)}exp(-\beta_2k^\frac{4}{3})& \text{if }\alpha=\frac{1}{3} \\ 
\frac{1}{2}U^2e^{\beta_2}k^{-\frac{5}{3}}exp(-\beta_2k^\frac{4}{3})exp(-\beta_4(k^{(2\alpha-\frac{2}{3})}-1 )) &\text{if } \alpha\in(0,1) \text{ and }\alpha\neq \frac{1}{3}  .
\end{cases}
$$
$E(k)$ decays exponentially in the dissipation range, which is consistent with the Kolmogorov scaling theory. In the dissipation range, $k\gg1$. Then, if $\alpha \in (\frac{1}{3}, 1)$, we have $\beta_4 > 0$ and $(k^{(2\alpha-\frac{2}{3})}-1)>0$ and thus $\beta_4(k^{(2\alpha-\frac{2}{3})}-1)>0$. Similarly, if $\alpha \in (0,\frac{1}{3})$, we have $\beta_4 < 0$ and $(k^{(2\alpha-\frac{2}{3})}-1)<0$ and $\beta_4(k^{(2\alpha-\frac{2}{3})}-1)>0$. So for all $\alpha\in (0,\frac{1}{3})\cup(\frac{1}{3}, 1)$, our closure model results in enhanced exponential decay in the dissipation range.

%%%%%%%%%%%%%%%%%%%
\section{First-order IMEX time-stepping methods}\label{imexmethod}

Most turbulent flows of engineering and scientific interest occur in bounded flow regions. Thus, we are more interested in computing turbulent flows on bounded domains. Accordingly, $\Omega\subset R^d$, $d=2,3$, denotes an open, bounded domain and consider the problem 
\begin{equation}\label{eq:model}
\left\{\begin{array}{rl}
    u_t + (u\cdot \nabla)u  -  \nu \Delta u +\gamma (-\Delta )^{\alpha} u+ \nabla p = f(t, x) &\text{in} \;\; (0,T] \times \Omega \\
    \nabla \cdot u = 0 \;   &     \text{in}\;\;   (0,T] \times \Omega \\
    u = 0   &     \text{on}\;  (0,T]\times R^d \backslash \Omega \\
    u = u_0(x)      &      \text{on} \; \{0\} \times \Omega. \\
\end{array}\right.
\end{equation}
Here we do not change the definition of the fractional Laplacian as defined in \eqref{ffll} but merely restrict its range to the bounded domain $\Omega$. Note that because the domain of integration in \eqref{ffll} is $R^d$, we impose, in \eqref{eq:model}, the constraint $u=0$ on the complement domain $R^d \backslash \Omega$ instead of on the boundary $\partial\Omega$ of $\Omega$ as is done in the PDE setting. 

We first consider the simple first-order implicit-explicit (IMEX) Euler time-stepping scheme given as follows. Given $ u^n$, find $u^{n+1}$ and $p^{n+1}$ satifying
\begin{equation}\label{IMEX-Euler}
\left\{\begin{aligned}
\frac{{u}^{n+1}-u^n}{\Delta t}+ (u^{n}\cdot\nabla) u^{n+1}-\nu \Delta u^{n+1} +\gamma (-\Delta)^{\alpha} u^{n+1}+\nabla p^{n+1}&=f^{n+1}\quad \text{ in } \Omega \\
 \nabla \cdot u^{n+1} &=0 \quad \text{ in } \Omega.
\end{aligned}\right.
\end{equation}
We prove, in \S\ref{imex1}, that this time-stepping scheme is unconditionally stable and first-order accurate.

Because of its implicit treatment of the fractional Laplacian term, the scheme \eqref{IMEX-Euler} requires the solution of a {\em dense} linear system at each time step. Having to also handle the Navier-Stokes terms makes for an even greater computational challenge. Thus, it is tempting to lag the fractional Laplacian term to the previous time step; however, this leads to serious stability issues so that that term has to be treated implicitly. However, there does exist a way to split the equations so that one can still solve a (modified) local momentum equation involving the usual {\em sparse matrices} and subsequently correct the solution by solving a linear nonlocal fractional Laplacian equation. Specifically, we propose to modify the IMEX Euler scheme \eqref{IMEX-Euler} into the following modular algorithm.

\begin{algorithm}[Modular IMEX Euler]\label{ModularAlgo}

\textit{Stage 1:  Given $ u^n$ in $X$, find $w^{n+1}$ in $X$ satisfying}
\begin{equation}\label{step1}
\left\{\begin{aligned}
\frac{{w}^{n+1}-u^n}{\Delta t}+(u^n\cdot\nabla)w^{n+1}
 -\nu \Delta  w^{n+1} +\nabla p^{n+1}&=f^{n+1}-\gamma (-\Delta)^\alpha u^n \quad \text{in } \Omega \\
\nabla \cdot w^{n+1} &=0 \quad \text{in } \Omega.
\end{aligned}\right.
\end{equation}

\textit{Stage 2:  Given $u^n$ and $w^{n+1}$ in $X$, find $u^{n+1}$ in $X$ satisfying}
\begin{gather}
2\Delta t\gamma(-\Delta )^\alpha (u^{n+1}-u^n)
+ u^{n+1}-w^{n+1}  =0 \quad \text{in } \Omega.\label{step2}
\end{gather}

\end{algorithm}
In \S\ref{imex2}, we prove that this time-stepping scheme is also unconditionally stable and first-order accurate. However, we can now solve the Stage 1 problem using a legacy Navier-Stokes code with the only modification necessary being in the construction of the right-hand side. Then, in the second stage, one solves a ``Poisson'' problem for the fractional Laplacian operator which involves a symmetric, positive definite, albeit dense linear system. This two-stage algorithm, although involving two linear system solves per time step, requires, compared to the algorithm given in \eqref{IMEX-Euler}, much less coding effort and introduces efficiencies not possible for the scheme \eqref{IMEX-Euler}.

{\allowdisplaybreaks
%%%%%%%%%%%%%%%%%%%%
\subsection{Preliminaries}
We first recall that for $\alpha\in (0,1)$, the fractional Sobolev space $W^{\alpha,p}(R^d)$ is defined as
$$
W^{\alpha,p}(R^d):=\left\{u\in L^p(R^d): \frac{\vert u(x)-u(y)\vert}{\vert x-y\vert^{\frac{d}{p}+\alpha}}\in L^p(R^d\times R^d)\right\}
$$
which is an intermediary Banach space between $L^p(R^d)$ and $W^{1,p}(R^d)$, equipped with the natural norm
$$
\Vert u\Vert_{W^{\alpha,p}(R^d)}:= \left( \int_{R^d}\vert u\vert^p dx + \int_{R^d}\int_{R^d} \frac{\vert u(x)-u(y)\vert^p}{\vert x-y\vert^{d+\alpha p}}dxdy\right)^{\frac{1}{p}}.
$$
For $p=2$, we have $W^{\alpha,2}(R^d)=H^{\alpha}(R^d)$ and 
$$
\Vert u\Vert_{H^{\alpha}(R^d)}=\Vert u\Vert_{W^{\alpha,2}(R^d)}= \left( \int_{R^d}\vert u\vert^2 dx + \int_{R^d}\int_{R^d} \frac{\vert u(x)-u(y)\vert^2}{\vert x-y\vert^{d+2\alpha}}dxdy\right)^{\frac{1}{2}}.
$$
We denote the Gagliardo (semi)-norm of $u$ by
$$
\vertiii {u}_\alpha= \left( \int_{R^d}\int_{R^d} \frac{\vert u(x)-u(y)\vert^2}{\vert x-y\vert^{d+2\alpha}}dxdy\right)^{\frac{1}{2}}.
$$
Define 
$$
\begin{aligned}
H^\alpha_{\Omega}(R^d)\text{ }&:=\{v\in H^\alpha(R^d)\text{ }:v=0 \text{ in } R^d \backslash \Omega \}\\
H^1_{\Omega}(R^d)\text{ }&:=\{v\in H^1(R^d)\text{ }:v=0 \text{ in } R^d \backslash \Omega \}\\
L_{0}^2(\Omega)&=\{q\in L^2(\Omega)\text{ }:\int_{\Omega}q \text{ }dx=0\}.
\end{aligned}
$$
Let $X$ denote the velocity space and $Q$ be the pressure space, defined by
$$
X\text{ }:=H^\alpha_{\Omega}(R^d)\cap H^1_{\Omega}(R^d),\text{ }Q\text{ }:=L_{0}^2(\Omega).
$$

\begin{remark}\label{rem2}
Generally, $H^1_{\Omega}(R^d) \nsubseteq H^\alpha_{\Omega}(R^d)$ and thus $H^\alpha_{\Omega}(R^d) \cap H^1_{\Omega}(R^d) \neq H^1_{\Omega}(R^d)$. This seems to cause some difficulties in finding a suitable finite element space that is a subset of $X$ to approximate solutions. However, in practical numerical simulations, one cannot integrate over all of $R^d$. So the domain of integration must be restricted. Here, for a $\lambda>0$, we assume $\Omega_\lambda$ is the interaction domain defined by
\begin{equation}\label{lambda}
\Omega_\lambda=\{y\in R^d\backslash \Omega\,\,:\,\, |x-y|\leq \lambda\,\, \text{for some } x\in \Omega \}
\end{equation}
and one only integrates over $\Omega\cup\Omega_\lambda$. Let
$$
\begin{aligned}
H^\alpha_{\Omega}(\Omega\cup\Omega_\lambda)\text{ }:=\{v\in H^\alpha(\Omega\cup\Omega_\lambda)\,\,:\,\,v=0 \text{ in } R^d \backslash \Omega \}\\
H^1_{\Omega}(\Omega\cup\Omega_\lambda)\text{ }:=\{v\in H^1(\Omega\cup\Omega_\lambda)\,\,:\,\,v=0 \text{ in } R^d \backslash \Omega \}.
\end{aligned}
$$
Then, to use finite element methods, one only needs to find a finite element space that is a subset of $X= H^\alpha_{\Omega}(\Omega\cup\Omega_\lambda) \cap H^1_{\Omega}(\Omega\cup\Omega_\lambda)=  H^1_{\Omega}(\Omega\cup\Omega_\lambda)$. Thus, the usual finite element spaces, such as continuous piecewise-quadratic elements, can be employed. It is shown in \S{\rm\ref{truncatedinteractions}} that the error incurred by domain truncation is of $O(1/\lambda^{2\alpha})$.
\end{remark}

The norm on the dual space of $X$ is defined by
$$
\Vert f\Vert _{-1}=\sup_{0\neq v\in X}\frac{\int_{\Omega}f v dx }{\Vert
\nabla v\Vert_{L^2(\Omega)} }.
$$
Define the usual skew symmetric trilinear form
$$
{b_\Omega}(u,v,w):=\int_{\Omega} (u\cdot\nabla) v \cdot w dx.
$$
For for all $u, v, w \in X$, we have the following inequalities \cite{Layton08,Shen92}:
$$
\begin{aligned}
{b_\Omega}(u,v,w)&\leq C \sqrt{\Vert \nabla u\Vert_{L^2(\Omega)}\Vert u\Vert_{L^2(\Omega)}}  \Vert\nabla v\Vert_{L^2(\Omega)}\Vert\nabla
w \Vert_{L^2(\Omega)}\\
{b_\Omega}(u,v,w)&\leq C \Vert u\Vert_{L^2(\Omega)} \Vert v\Vert_{H^2(\Omega)}\Vert\nabla
w \Vert_{L^2(\Omega)}.
\end{aligned}
$$

%%%%%%%%%%%%%%%%%
\subsection{Analysis for the algorithm {\bf(\ref{IMEX-Euler})}}\label{imex1}

We prove that the time-stepping scheme \eqref{IMEX-Euler} is unconditionally stable and first-order convergent. We do not provide the proofs because they are similar to those for the theorems considered in \S\ref{imex2}.

\begin{theorem}[Unconditional stability of time-stepping scheme \eqref{IMEX-Euler}]\label{th:IMEX-Euler}
The IMEX Euler scheme  \eqref{IMEX-Euler} is unconditionally, long-time stable. Specifically, for any $N\geq 1$, we have that
$$%\begin{equation}\label{thmineq:IMEX-Euler}
\begin{aligned}
\frac{1}{2}\Vert u^{N}\Vert_{L^2(\Omega)}^2+&\sum_{n=0}^{N-1}\frac{1}{2} \Vert u^{n+1}-u^{n} \Vert_{L^2(\Omega)}^2+\sum_{n=0}^{N-1}\frac{1}{2}\Delta t\nu\Vert \nabla u^{n+1}\Vert_{L^2(\Omega)}^2  \\
& +\sum_{n=0}^{N-1}\Delta t\gamma \frac{C_{d,\alpha}}{2}\vertiii{ u^{n+1}}^2_\alpha
 \leq\frac{1}{2}\Vert u^{0}\Vert_{L^2(\Omega)}^2+\sum_{n=0}^{N-1}\frac{\Delta t}{2\nu}\Vert f^{n+1} \Vert_{-1}^2. 
\end{aligned}
$$%\end{equation}

\end{theorem}

\begin{comment}
\begin{proof}
Take inner product of \eqref{IMEX-Euler} with
$u^{n+1}$ and multiply through by $\Delta t$. This
gives, 
\begin{equation}\label{eq:IMEX-Euler} 
\begin{aligned}
\frac{1}{2}\Vert u^{n+1}\Vert_{L^2(\Omega)}^2-&\frac{1}{2}\Vert u^{n}\Vert_{L^2(\Omega)}^2+\frac{1}{2} \Vert u^{n+1}-u^{n} \Vert_{L^2(\Omega)}^2 \\
&+\Delta t\nu\Vert \nabla u^{n+1}\Vert_{L^2(\Omega)}^2 +\Delta t\gamma \frac{C_{d,\alpha}}{2}\vertiii{ u^{n+1}} ^2_\alpha=\Delta t \int_{\Omega} f^{n+1}u^{n+1}dx. 
 \end{aligned}
\end{equation}

\noindent Applying Young's inequality to the RHS, \eqref{eq:IMEX-Euler} reduces to
\begin{equation}\label{ineq:IMEX-Euler1} 
\begin{aligned}
\frac{1}{2}\Vert u^{n+1}\Vert_{L^2(\Omega)}^2-&\frac{1}{2}\Vert u^{n}\Vert_{L^2(\Omega)}^2+\frac{1}{2} \Vert u^{n+1}-u^{n} \Vert_{L^2(\Omega)}^2 \\
&+\frac{1}{2}\Delta t\nu\Vert \nabla u^{n+1}\Vert_{L^2(\Omega)}^2 +\Delta t\gamma \frac{C_{d,\alpha}}{2}\vertiii{ u^{n+1}}^2_\alpha\leq\frac{\Delta t}{2\nu}\Vert f^{n+1} \Vert_{-1}^2. 
\end{aligned}
\end{equation}
 Summing up \eqref{ineq:IMEX-Euler1} from $n=1$ to $N-1$ results in \eqref{thmineq:IMEX-Euler}.
\end{proof}
\end{comment}

%\subsection{Error Analysis}
To analyze the rate of convergence of the approximation we assume the following regularity on the exact solution $(u,p)$ and the body force $f$ of \eqref{eq:model}:
\begin{gather*}
{u}\in L^{\infty}(0,T;H^{1}(\Omega))\cap
H^{2}(0,T;L^{2}(\Omega)),\\
{p} \in L^{2}(0,T;H^{1}(\Omega)),\quad \text{and}\quad f \in L^{2}%
(0,T;L^{2}(\Omega)).
\end{gather*}

Let $t^{n}=n\Delta t,n=0,1,2,...,N_T,$ and $T:=N_T\Delta t$. Denote $v^{n}=v(t^{n})$. We introduce the discrete norms
\[
\Vert|v|\Vert_{m,k,\Omega}:=\big(\sum_{n=0}^{N_T}||v^{n}||_{H^k(\Omega)}^{m}\Delta t\big)^{1/m}%
\qquad\mbox{and}\qquad\Vert|v|\Vert_{\infty,k,\Omega}=\max\limits_{0\leq n\leq N_T}\Vert v^{n}\Vert_{H^k(\Omega)}.
\]

Let $e^{n}={u}(t_n)-u^n$ denote the error at the time $t_n$ between the exact solution of \eqref{eq:model} and the approximation obtained using the time-stepping scheme \eqref{IMEX-Euler}. Then, we have the following error estimate.

\begin{theorem}
[Unconditional convergence of algorithm \eqref{IMEX-Euler}]\label{th:errIMEX-Euler} For any $0\leq t_{N}\leq T$, there exits a positive constant $C$
independent of the time step $\Delta t$ such that
\begin{equation}
\begin{aligned}
\frac{1}{2}\Vert e^{N}\Vert_{L^2(\Omega)}^{2}+&\frac{1}{8}\Delta t\nu
\Vert \nabla e^{N}\Vert^{2}_{L^2(\Omega)}+\frac{1}{2}\sum_{n=0}^{N-1}\Vert e^{n+1}-e^n\Vert_{L^2(\Omega)}^{2}\\
&+\frac{1}{2}\Delta t\sum_{n=0}^{N-1}\nu
\Vert \nabla e^{n+1}\Vert^{2}_{L^2(\Omega)}+\Delta t\sum_{n=0}^{N-1}\gamma \frac{C_{d,\alpha}}{2} \vertiii{e^{n+1}}_\alpha^2
\leq C\Delta t^2.
\end{aligned}
\end{equation}
\end{theorem}

\subsection{Analysis for the modular algorithm {\bf(\ref{step1})-(\ref{step2})}}\label{imex2}

We first prove that the modular time-stepping scheme \eqref{step1}-\eqref{step2} is unconditionally stable. We emphasize that the coefficient $2$ multiplying the fractional Laplacian term in \eqref{step2} is essential for the unconditional stability of the modular algorithm.

\begin{theorem}[Unconditional stability of the modular algorithm \eqref{step1}-\eqref{step2}]\label{th:ModularIMEXEuler}
The modular algorithm \eqref{step1}-\eqref{step2} is unconditionally, long-time stable. Specifically, for any $N\geq 1$, we have that
\begin{equation}\label{thmineq:ModularFracBE}
\begin{aligned}
\Vert u^{N}\Vert_{L^2(\Omega)}^2&+ \sum_{n=1}^{N-1}\Vert w^{n+1}- u^n\Vert_{L^2(\Omega)}^2
+\sum_{n=0}^{N-1}\Delta t \nu \Vert \nabla w^{n+1}\Vert_{L^2(\Omega)} ^2+\Delta t\gamma C_{d,\alpha}\vertiii{ u^{N}}^2_\alpha \\
&\qquad +\sum_{n=0}^{N-1}\Delta t\gamma \frac{3C_{d,\alpha}}{4}\vertiii {u^{n+1}-u^n}^2_\alpha+\sum_{n=0}^{N-1}\Delta t\gamma \frac{C_{d,\alpha}}{4}\vertiii {u^{n+1}+u^n}^2_\alpha\\
&\qquad+2\Delta  t^2\gamma^2\Vert (-\Delta)^\alpha u^{N}\Vert_{L^2(\Omega)}^2+\sum_{n=0}^{N-1}\frac{1}{2} \Vert w^{n+1}-u^{n+1}\Vert^2_{L^2(\Omega)}\\
&\leq \Vert u^0 \Vert_{L^2(\Omega)}^2 +\Delta t\gamma C_{d,\alpha}\vertiii{ u^{0}}^2_\alpha+2\Delta  t^2\gamma^2\Vert (-\Delta)^\alpha u^0\Vert_{L^2(\Omega)}^2\\
&\qquad+\sum_{n=0}^{N-1}\frac{\Delta t}{\nu}\Vert f^{n+1} \Vert_{-1}^2.
\end{aligned}
\end{equation}
\end{theorem}

\begin{proof}
The proof is provided in Section \ref{thm3.5proof}.
\end{proof}

%%%%%%%%%%%%%%%%%%%%%%%%%
%\subsubsection{Error analysis for the modular algorithm}

We next prove that the modular algorithm \eqref{step1}-\eqref{step2} is also a first-order in time scheme. To analyze the rate of convergence of the approximation we assume that the true solution ${u}$ and ${p}$ and the body force $f$ of \eqref{eq:model} have regularity given by
\begin{gather*}
{u}\in L^{\infty}(0,T;H^{2}(\Omega))\cap
H^{2}(0,T;L^{2}(\Omega)),\quad
(-\Delta)^\alpha u\in  H^{1}(0,T;H^{\alpha}(R^d)),\\
{p} \in L^{2}(0,T;H^{1}(\Omega)),\quad \text{and} \quad f \in L^{2}%
(0,T;L^{2}(\Omega)).
\end{gather*}

 Let $w^{n+1}$ and $ u^{n+1}$ denote the solution to Stages 1 and 2 in algorithm \eqref{step1}-\eqref{step2}, respectively. Denote $e^{n}={u}(t_n)-u^n$ and  $\bar{e}^{n}={u}(t_n)-w^n$. Then, we have the following error estimate.

{\allowdisplaybreaks

\begin{theorem}
[Unconditional first-order convergence of the modular algorithm \eqref{step1}-\eqref{step2}]\label{thm:conv+Modu} 
Consider the modular algorithm \eqref{step1}-\eqref{step2}. Assuming $\Delta t\leq 1$, then for any $0\leq t_{N}\leq T$, there is a positive constant $C$
independent of time step $\Delta t$ and mesh size $h$ such that

\begin{equation}\label{err000}
\begin{aligned}
\Vert e^{N}\Vert^2_{L^2(\Omega)}+&\sum_{n=0}^{N-1}\Vert \bar{e}^{n+1}-e^n\Vert^2_{L^2(\Omega)}+\sum_{n=0}^{N-1}\Vert e^{n+1}-\bar{e}^n\Vert^2_{L^2(\Omega)}\\
&+\Delta t\sum_{n=0}^{N-1}\nu \Vert \nabla \bar{e}^{n+1}\Vert^2_{L^2(\Omega)}
+\Delta t\gamma C_{d,\alpha}\vertiii{e^N}_\alpha^2\\
&+\Delta t\sum_{n=0}^{N-1}\gamma \frac{3C_{d,\alpha}}{4}\vertiii{e^{n+1}-e^n}_\alpha^2
+\Delta t\sum_{n=0}^{N-1}\gamma \frac{C_{d,\alpha}}{4}\vertiii{e^{n+1}+e^n}_\alpha^2\\
&+2\Delta  t^2\gamma^2\Vert (-\Delta)^\alpha e^{N}\Vert_{L^2(\Omega)}^2
\leq C\Delta t^2.
\end{aligned}
\end{equation}
\end{theorem}

\begin{proof}
The proof is provided in Section \ref{thm3.6proof}.
\end{proof}

}

%%%%%%%%%%%%%%%%%%%%%%%%
\section{Truncation of interactions}\label{truncatedinteractions}

As already stated in Remark \ref{rem2}, it is impractical to integrate over all of $R^d$. Furthermore, it is reasonable to assume that the velocity and pressure fields at points that are far away from a point $x$ in the bounded domain of interest $\Omega$ have negligible effect on the velocity and pressure fields at points close to $x$. This is clearly true for the integrand kernels we consider in our model. Thus, we consider limiting the extent of nonlocal interactions for a point $x\in\Omega$ to the ball of radius $\lambda$ centered at $x\in\Omega$ denoted by
$$
\mathcal{B}_{\lambda}(x):=\{y\in \Omega \cup \Omega_\lambda: |y-x|\leq \lambda \}
$$
and define the interaction domain $\Omega_\lambda$ as in \eqref{lambda}.
%\Omega_\lambda:=\{y\in R^d\backslash \Omega\,\,:\,\, |x-y|\leq \lambda\,\, \text{\color{black}for some } x\in \Omega \},
%$$
We then consider the truncated variational problem
\begin{equation}\label{prob}
\left\{\begin{aligned}
 \int_{\Omega }\frac{\partial{u}_\lambda}{\partial t} v dx& +{b_\Omega} ({u}_\lambda, {u}_\lambda, v)+\nu \int_{\Omega} \nabla {u}_\lambda : \nabla v dx-\int_{\Omega} {p}_\lambda (\nabla\cdot v)dx\\
 &+\gamma \frac{C_{d,\alpha}}{2}\int_{\Omega \cup \Omega_\lambda }\int_{(\Omega\cup\Omega_\lambda)\cap\mathcal{B}_{\lambda}(x)} \frac{{u}_\lambda(x)-{u}_\lambda(y)}{|x-y|^{n+2\alpha}}(v(x)-v(y))dydx
 \\
 &\qquad\qquad= \int_{\Omega} f v dx \quad\forall\, v\in X\\
 \int_{\Omega} (\nabla\cdot {u}_\lambda) q dx&=0 \quad\forall\, q \in Q.
\end{aligned}\right.
\end{equation}
{\color{black}The nonlocal term in \eqref{prob} now involves a finite integration domain, in contrast to that in \eqref{eq:model} that features an infinite integration domain.}

{\color{black}We recall that for $\alpha\in (0,1)$, the fractional Sobolev space $W^{\alpha,p}(\Omega)$ is defined as
$$
W^{\alpha,p}(\Omega):=\bigg\{u\in L^p(\Omega): \frac{\vert u(x)-u(y)\vert}{\vert x-y\vert^{\frac{d}{p}+\alpha}}\in L^p(\Omega\times \Omega)\bigg\}
$$
equipped with the natural norm
$$
\Vert u\Vert_{W^{\alpha,p}(\Omega)}:= \left( \int_{\Omega}\vert u\vert^p dx + \int_{\Omega}\int_{\Omega} \frac{\vert u(x)-u(y)\vert^p}{\vert x-y\vert^{d+\alpha p}}dxdy\right)^{\frac{1}{p}}.
$$
For $p=2$, we have
$$
\Vert u\Vert_{H^{\alpha}(\Omega)}=\Vert u\Vert_{W^{\alpha,2}(\Omega)}= \left( \int_{\Omega}\vert u\vert^2 dx + \int_{\Omega}\int_{\Omega} \frac{\vert u(x)-u(y)\vert^2}{\vert x-y\vert^{d+2\alpha}}dxdy\right)^{\frac{1}{2}}.
$$
We define the following norm of $u$ by
$$
\vertiii {u}_{\alpha, \Omega\cup \Omega_{\lambda}
}= \left( \int_{\Omega\cup\Omega_{\lambda}}
\int_{(\Omega\cup\Omega_{\lambda})\cap B_{\lambda}(y)} \frac{\vert u(x)-u(y)\vert^2}{\vert x-y\vert^{d+2\alpha}}dxdy\right)^{\frac{1}{2}}.
$$

}

We define the velocity and pressure spaces 
$$
X:=H^\alpha_{\Omega}(\Omega\cup\Omega_\lambda)\cap H^1_{\Omega}(\Omega\cup\Omega_\lambda)= H^1_{\Omega}(\Omega\cup\Omega_\lambda)\quad\mbox{and}\quad
Q:=L_{0}^2(\Omega),
$$
respectively, where
$$
H^\alpha_{\Omega}(\Omega\cup\Omega_\lambda)=\{v\in H^\alpha(\Omega\cup\Omega_\lambda)\,\,:\,\, v=0 \text{ in } R^d \backslash \Omega \}
$$
and
$$
H^1_{\Omega}(\Omega\cup\Omega_\lambda):=\{v\in H^1(\Omega\cup\Omega_\lambda)
\,\,:\,\,
v=0 \text{ in } R^d \backslash \Omega \}.
$$

Let $V$ denote divergence free velocity space
$$
V:=\{ v\in X\,\,:\,\, (\nabla\cdot v, q)=0 \text{ }\forall\, q\in Q\}.
$$
 
%\subsection{Truncation error estimate} 

In the next theorem, we analyze the error due to domain truncation and the rate of convergence with respect to $\lambda$. We assume that the true solution ${u}$ of \eqref{eq:model} has regularity given by
$$
{u}\in L^{2}(0,T;L^{2}(\Omega))\cap L^{\infty}(0,T; H^1(\Omega)).
$$
Also, we denote the error due to truncation by $e_\lambda=u-u_\lambda$, where $u$ and $u_\lambda$ denote the solutions of \eqref{eq:model} and \eqref{prob}, respectively.\footnote{For the simpler setting of the fractional Laplacian Poisson problem, the error due to truncation was considered in \cite{Max13}.}

\begin{theorem}
[Error due to domain truncation]
\label{thm:trancate} 
Let $I: =\min  \{ R: \Omega \subset B_R(x)\, \forall\, x\in \Omega\}$. Then, for any $0\leq t\leq T$, there exists a positive constant $C$
independent of the time step $\Delta t$ and the truncation radius $\lambda$ such that
\begin{equation}
\begin{aligned}
\Vert e_\lambda(x,t)\Vert^2_{L^2(\Omega)}+\int_{0}^t \frac{1}{16}\nu \Vert\nabla e_\lambda(x,x^\prime) &\Vert^2_{L^2(\Omega)} dx^\prime+\int_{0}^t \gamma \frac{C_{d,\alpha}}{2}\vertiii {e_\lambda(x,x^\prime)}_{\alpha, \Omega\cup\Omega_{\lambda}} dx^\prime\\
&\leq  C\left[ \left(\frac{1}{\lambda^{2\alpha}}\right)^2 + \left(\frac{(\lambda+I)^d}{\lambda^{d+2\alpha}} \right)^2+
\left(\frac{1}{\lambda^{d+2\alpha}}\right)^2 \right].
\end{aligned}
\end{equation}
In particular, if $\lambda\geq 1$, we have
\begin{equation}
\begin{aligned}
&\Vert e_\lambda(x,T)\Vert^2_{L^2(\Omega)}+\int_{0}^T \frac{1}{16}\nu \Vert\nabla e_\lambda(x,x^\prime) \Vert^2_{L^2(\Omega)} dx^\prime\\
&\qquad\qquad\qquad+\int_{0}^T \gamma \frac{C_{d,\alpha}}{2}\vertiii {e_\lambda(x,x^\prime)}_{\alpha, \Omega\cup\Omega_{\lambda}} dx^\prime
\leq  C\left(\frac{1}{\lambda^{2\alpha}}\right) ^2.
\end{aligned}
\end{equation}

\end{theorem}

\begin{proof}
The proof is provided in Section \ref{thm4.1proof}.
\end{proof}

%%%%%%%%%%%%%%%%%
\section{Finite element approximations}\label{sec:femfem}

We denote by $X_h\subset X$ and $ Q_h\subset Q$ conforming velocity and pressure finite element spaces, respectively, based on an edge-to-edge triangulation of $\Omega\cup\Omega_\lambda$ with maximum triangle diameter $h$ and with $\partial\Omega$ consisting of triangle vertices and/or edges. We assume that $X_h$ and $Q_h$ satisfy the usual discrete inf-sup condition and the approximation properties, \cite{Layton08}
\begin{align}
\inf_{v_h\in X_h}\| v- v_h \|_{L^2(\Omega)}&\leq C h^{k+1}\Vert u \Vert_{H^{k+1}(\Omega)}&\forall v\in [H^{k+1}(\Omega)]^d\label{Interp1}\\
\inf_{v_h\in X_h}\| \nabla ( v- v_h )\|_{L^2(\Omega)}&\leq C h^k \Vert v\Vert_{H^{k+1}(\Omega)}&\forall v\in [H^{k+1}(\Omega)]^d\label{interp2}\\
\inf_{q_h\in Q_h}\|  q- q_h \|_{L^2(\Omega)}&\leq C h^{s+1}\Vert p\Vert_{H^{s+1}(\Omega)}&\forall q\in H^{s+1}(\Omega)\label{interp3}
\end{align}
for a constant $C>0$ having value independent of $h$. For instance, the commonly used Taylor-Hood $P^{s+1}$-$P^{s}$, $s\geq 1$, element pairs, \cite{G89}, satisfy both the discrete LBB condition and the approximation properties \eqref{Interp1}-\eqref{interp3}.

We define the usual explicitly skew symmetric trilinear form
\[
{\widetilde b}_{\Omega}(u,v,w):=\frac{1}{2}(u\cdot\nabla v,w)-\frac{1}{2}(u\cdot\nabla w,v).
\]
which satisfies the bound \cite{Layton08}
\begin{gather}
{\widetilde b}_{\Omega}(u,v,w)\leq C \left(\Vert \nabla u\Vert\Vert u\Vert\right)^{1/2}\Vert\nabla v\Vert\Vert\nabla
w \Vert\text{ for all } u, v, w \in X .\label{In1}
\end{gather}

We use the first-order IMEX Euler time stepping scheme presented in Section \ref{imexmethod} for time discretization and the Taylor-Hood P2-P1 element pair for spatial discretization. Let $u_{h,\lambda}^0$ denote an approximation to $u_0$, e.g., the $X_h$-interpolant of the initial datum $u_0$. Then, the fully discrete approximation of \eqref{prob} is given as follows. For $n=0,1,\ldots,N=T/\Delta t$, given $u_{h,\lambda}^n$, find $u_{h,\lambda}^{n+1}\in X_h$ and $p_{h,\lambda}^{n+1}\in Q_h$ satisfying
\begin{equation}\label{full}
\left\{\begin{aligned}
&\int_{\Omega}\frac{u_{h,\lambda}^{n+1}-u_{h,\lambda}^n}{\Delta t} v dx +{\widetilde b}_{\Omega}(u_{h,\lambda}^n,u_{h,\lambda}^{n+1},v)+\nu\int_\Omega \nabla u_{h,\lambda}^{n+1}:\nabla v dx 
\\
&+\gamma\frac{C_{d,\alpha}}{2}\int_{\Omega\cup
\Omega_{\lambda}}\int_{\Omega\cup
\Omega_\lambda\cap
\mathcal{B}_\lambda(x)}\frac{u_{h,\lambda}^{n+1}(x)-u_{h,\lambda}^{n+1}(x^\prime)}{|x-x^\prime|^{d+2\alpha}}(v(x)-v(x^\prime))dx^\prime dx
\\
&\qquad\qquad\qquad\qquad\qquad\qquad-\int_{\Omega}p_{h,\lambda}^{n+1}(\nabla\cdot v)dx=\int_{\Omega} f^{n+1} v dx\quad\forall v\in X_h
\\
&\int_{\Omega}\nabla\cdot u_{h,\lambda}^{n+1}qdx=0\quad \forall q\in Q_h.
\end{aligned}\right.
\end{equation}
Gaussian quadrature rules are used for approximating the integrals involved in \eqref{full} and the CLapack software suite \cite{lapack} is used for solving the resulting linear systems. Of course, domain truncation introduces an error additional to the spatial/temporal discretization error; see Theorem \ref{thm:trancate}.

%%%%%%%%%%%%%%%%%%%%%%%%%%%%
\subsection{Error analysis for the fully discrete scheme}\label{errorest}

The focus of this section is the analysis on temporal accuracy of the time stepping method and the spatial accuracy associated with the finite element approximation for the proposed method.

% To analyze the rate of convergence of the approximation we assume the following regularity on the true solution of \eqref{eq:model}
%\begin{gather*}
%\hat{u}\in L^{\infty}(0,T;H^{1}(\Omega))\cap
%H^{2}(0,T;L^{2}(\Omega)),\\
%\hat{p} \in L^{2}(0,T;H^{1}(\Omega)), \text{and }f \in L^{2}%
%(0,T;L^{2}(\Omega)).
%\end{gather*}
Let ${u_\lambda^{n}}={u_\lambda}(t^{n})$ and  ${p_\lambda^{n}}={p_\lambda}(t^{n})$, where $({u_\lambda}, {p_\lambda})$ denotes the solution of the truncated variational problem \eqref{prob}. To analyze the rate of convergence of the approximation, we assume that ${u_\lambda}$ and ${p_\lambda}$ and the body force $f$ of \eqref{prob} satisfy
\begin{gather*}
{u_\lambda}\in L^{\infty}(0,T; H^1(\Omega))\cap L^{4}(0,T;H^{k+1}(\Omega))\cap
H^{2}(0,T;L^{2}(\Omega))\cap H^{1}(0,T;H^{k+1}(\Omega)),\\
{p_\lambda} \in L^{2}(0,T;H^{s+1}(\Omega)),\quad \text{and} \quad f \in L^{2}%
(0,T;L^{2}(\Omega)).
\end{gather*}
Let $e_{h,\lambda}^n={u_\lambda^{n}}-u_{h,\lambda}^n$ denote the error between the solution of truncated variational problem \eqref{prob} and the approximation obtained from the fully discrete scheme \eqref{full}. Then, we have the following error estimate.

\begin{theorem}
[Convergence of the fully discrete scheme \eqref{full}]\label{fullydiscreteerror}
For any $0\leq t^{N}\leq T$, there exists a positive constant $C$
independent of the time step $\Delta t$ and mesh size $h$ such that
\begin{equation}\label{full-err}
\begin{aligned}
&\frac{1}{2}\Vert e_{h,\lambda}^{N}\Vert_{L^2(\Omega)}^{2}
+\frac{1}{2}\Delta t\sum_{n=0}^{N-1}\nu
\Vert \nabla e_{h,\lambda}^{n+1}\Vert^{2}_{L^2(\Omega)}
\\
&\qquad\qquad\qquad\leq C\big(\|\nabla
e_{h,\lambda}^{0}\|^{2}_{L^2(\Omega)}+\Vert e_{h,\lambda}^{0}\Vert_{L^2(\Omega)}^{2}
+h^{2k}+h^{2s+2}+\Delta t^2\big).
\end{aligned}
\end{equation}
\end{theorem}

\begin{proof}
The proof is provided in Section \ref{thm5.1proof}.
\end{proof}

For the $P^2-P^1$ Taylor-Hood element pair ($k=2, s=1,$), i.e., the $C^{0}$ piecewise
quadratic velocity space $X_{h}$ and $C^{0}$ the piecewise linear pressure space
$Q_{h}$, we have the following estimate.

\begin{corollary}
Assuming that $(X_{h},Q_{h}$) is given
by the $P^2-P^1$ Taylor-Hood element pair and that $\Vert e_{h,\lambda}^{0}\Vert_{L^2(\Omega)}$ and $\Vert\nabla e_{h,\lambda}^0\Vert_{L^2(\Omega)}$ are both at least  $O(h^2)$ accurate, we have
\begin{gather}
\frac{1}{2}\Vert e_{h,\lambda}^{N}\Vert_{L^2(\Omega)}^{2}+\Delta t\sum_{n=0}^{N-1}\frac{\nu}{2}\Vert\nabla
e_{h,\lambda}^{n+1}\Vert_{L^2(\Omega)}^{2} 
\leq C (h^4 + \Delta t^2 )\text{ .}
\end{gather}

\end{corollary}

Combining with the result of Theorem \ref{thm:trancate} concerning the error due to domain truncation, we have the following result.

\begin{theorem}\label{th:original-err}
Let $u(t^n)$ denote the solution of the original turbulence model \eqref{eq:model}, ${u_\lambda}(t^n)$ denote the solution of the truncated variational problem \eqref{prob}, and $u^n_{h,\lambda}$ denote the solution of the fully discrete problem \eqref{full}. Assume that $(X_{h},Q_{h}$) is given
by the $P^2-P^1$ Taylor-Hood element pair and that $\Vert e_{h,\lambda}^{0}\Vert_{L^2(\Omega)}$ and $\Vert\nabla e_{h,\lambda}^0\Vert_{L^2(\Omega)}$ are both at least  $O(h^2)$ accurate. Then, for any $0\leq t^{N}\leq T$, there exists a positive constant $C$
independent of time step $\Delta t$ and the mesh size $h$ such that
\begin{align*}
\Vert u(t^n)-u^n_{h,\lambda}\Vert_{L^2(\Omega)}^2&=\Vert u(t^n)-{u_\lambda}(t^n)\Vert_{L^2(\Omega)}^2+\Vert {u_\lambda}(t^n)-u^n_{h,\lambda}\Vert_{L^2(\Omega)}^2\\
&\leq C (\frac{1}{\lambda^{4\alpha}}+h^4+\Delta t^2).
\end{align*}
\end{theorem}

%%%%%%%%%%%%%%%%%%%%%%%
\subsection{Convergence study}

We illustrate the results of Section \ref{errorest} by considering the manufactured solution
\begin{equation}\label{solution}
u(x,y,t)=x^2y+\sin(t),\quad
v(x,y,t)=-xy^2+\sin(t),\quad
p(x,y,t)=\sin(x)+\sin(y)
\end{equation}
over a square domain $\Omega=[0,1]^2.$ The initial condition $u_0$ and the right-hand side function $f$ are determined by substituting the exact solution \eqref{solution} into \eqref{prob}. Note that now the volume constraint is inhomogeneous and is also set to the exact solution \eqref{solution}. We focus on the errors due to the temporal and finite element discretizations; for illustrations of the errors due to domain truncation, see \cite{Max13}. In addition, we examine convergence behaviors with respect to other norms that are not covered by the theoretical results of Section \ref{errorest}. 

Errors of the approximate solutions are given in Figure \ref{fig:error} for $\Delta t=h$ (left) and $\Delta t=h^2$ (right), respectively, where $h$ denotes the mesh size and $\Delta t$ the time step. The time interval of interest is set to $[0,0.5]$ and the horizon parameter is set to $\lambda=1.$ First, for errors measured in the $H^1(\Omega)$ norm for the two components $u_1$ and $u_2$ of the velocity, one observes the expected linear convergence for $\Delta t=h$ and quadratic convergence for $\Delta t=h^2$.  For $L^2(\Omega)$ errors for the pressure $p$, linear convergence is also observed if $\Delta t=h^2$. In addition, for $\Delta t=h^2$, we observe that the errors in the $L^2(\Omega)$ and $L^\infty(\Omega)$ velocity components seem to exhibit near cubic convergence. Quantitative information along these lines is given in Tables \ref{tab:converge_h} and \ref{tab:converge_h2}.

Clearly, the efficacy of the turbulence closure model studied in this paper needs to be demonstrated through computational testing in more realistic settings. This is the subject of future work.

%% nonlocal, gamma=1, labmda = 1
\begin{center}
\begin{figure}[h!]\centering\footnotesize
$$
\begin{array}{cc}
\includegraphics[width=.475\textwidth]{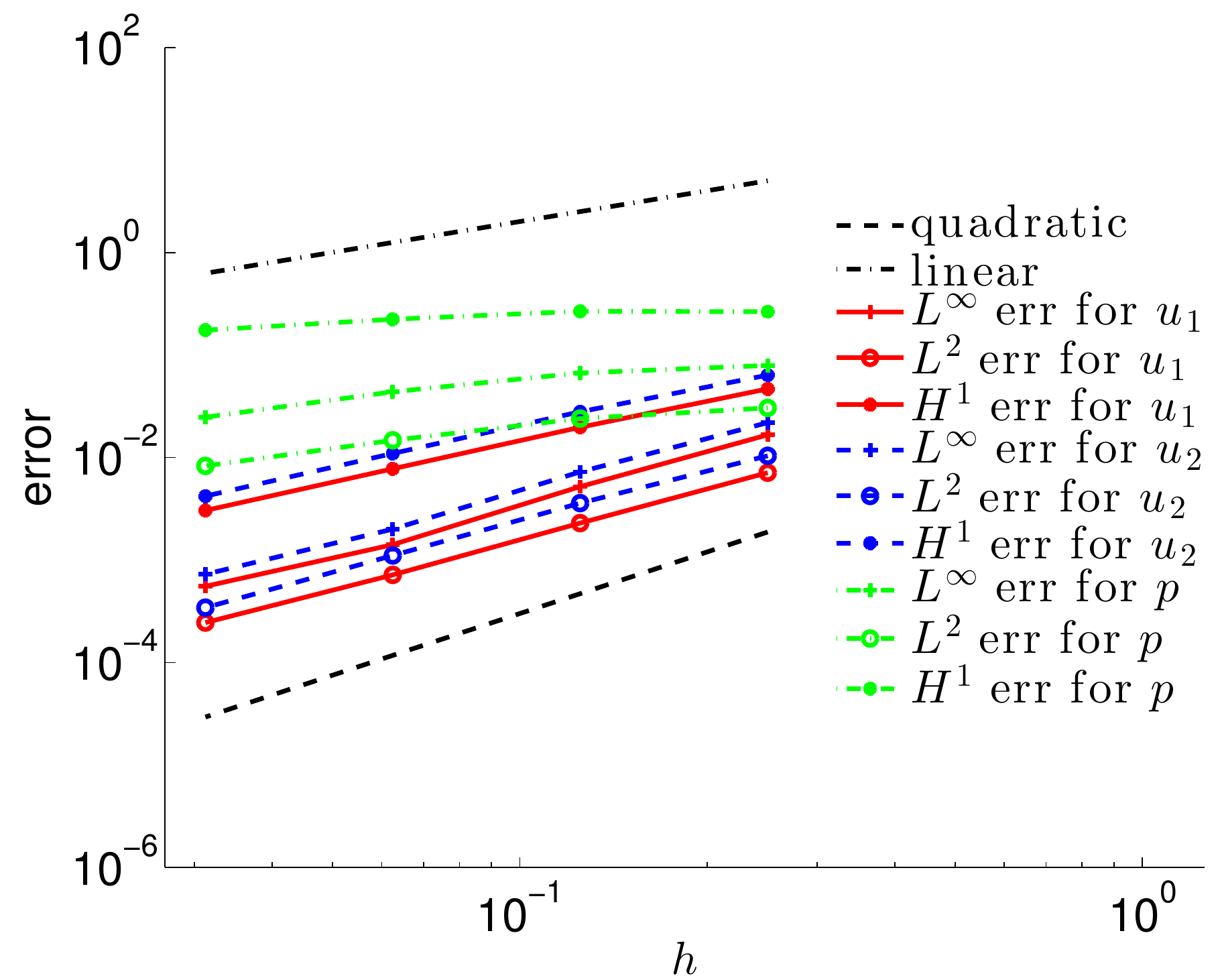}
&
\includegraphics[width=.475\textwidth]{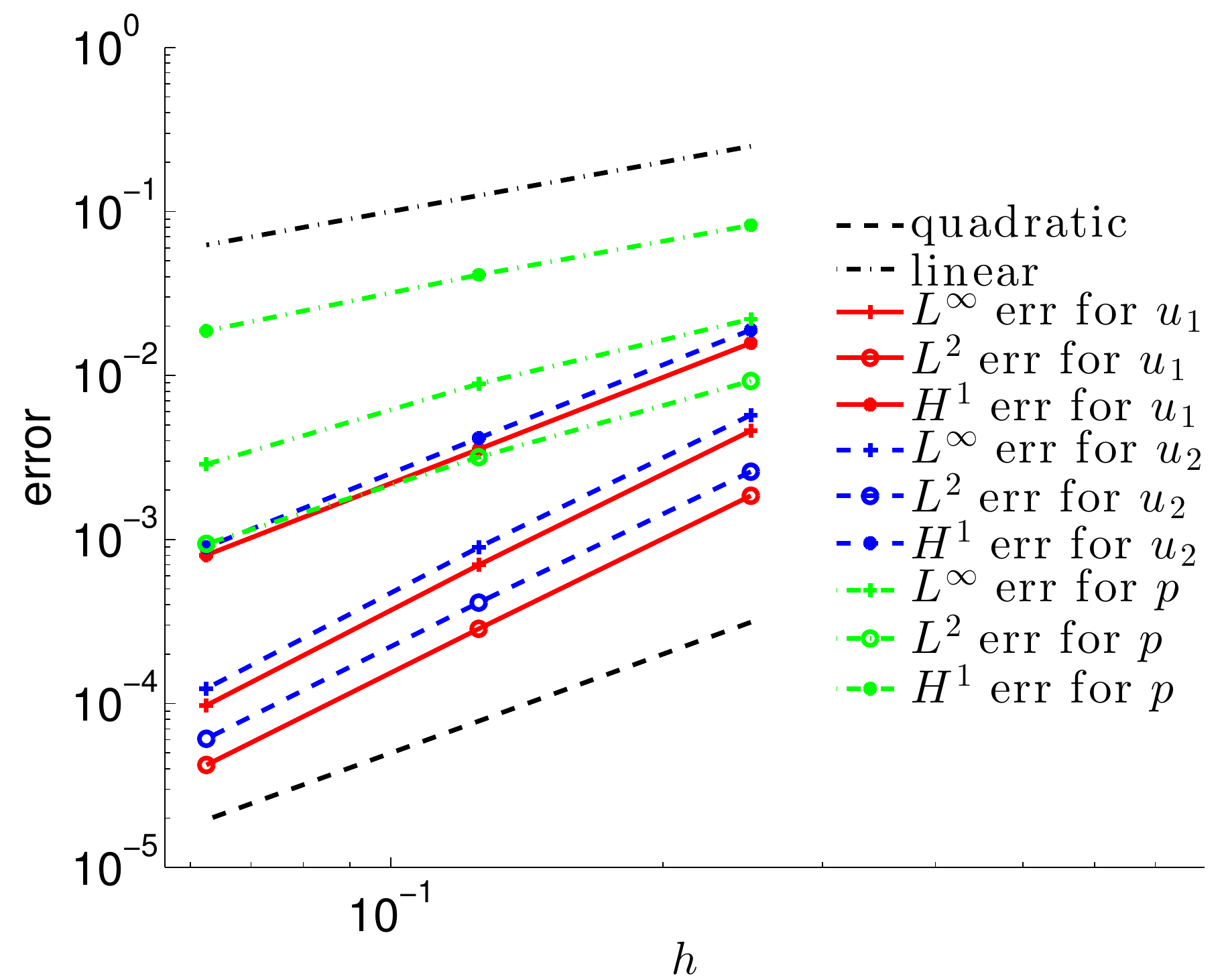}
\\
\Delta t=h
&
\Delta t=h^2
\end{array}
$$
\caption{Error behaviors of solutions of the fully discretized system \eqref{full} with  $\gamma=1$ and $\lambda=1$ for the exact solution \eqref{solution}.}\label{fig:error}
\end{figure}
\end{center}

\begin{table}[h!]
\caption{Computational results for $\Delta t=h$ for the exact solution \eqref{solution}.}\label{tab:converge_h}
\begin{center}\footnotesize
\renewcommand{\arraystretch}{1.3}
\begin{tabular}{|c|c|c||c|c||c|c|c|c||c|c||c|c||c|c||c|c||c|c|}\hline 
& \multicolumn{6}{c|}{errors for variable $u_1$}\\\hline
$h$&$L^\infty$ error &rate&$L^2$ error&rate&$H^1$ error&rate\\\hline
1/4&1.6579e-02 &--&           	7.1046e-03 &--&           	4.6749e-02&  -- \\\hline
1/8&5.1860e-03 &      1.6767 &	2.2907e-03 &      1.6329 &	1.9687e-02 &     1.2477	\\\hline
1/16&1.4107e-03&      1.8782 & 	7.1371e-04&      1.6824&	7.7386e-03&      1.3471	\\\hline
1/32&5.5281e-04&      1.3516&	2.4519e-04&      1.5415&	3.0393e-03&      1.3484	\\\hline
& \multicolumn{6}{c|}{errors for variable $u_2$}\\\hline
$h$&$L^\infty$ error &rate&$L^2$ error&rate&$H^1$ error&rate\\\hline
1/4&2.1887e-02&--&            	1.0398e-02&--&            	6.3696e-02 &\\\hline
1/8&7.2414e-03&      1.5957&	3.5802e-03&      1.5382&	2.7841e-02&  --    1.1940\\\hline
1/16&2.0032e-03&      1.8539&	1.1045e-03&      1.6966&	1.0931e-02&      1.3488	\\\hline
1/32&7.2338e-04&      1.4695&	3.4263e-04&      1.6887&	4.1859e-03&      1.3848	\\\hline
& \multicolumn{6}{c|}{errors for variable $p$}\\\hline
$h$&$L^\infty$ error &rate&$L^2$ error&rate&$H^1$ error&rate\\\hline
1/4&7.9094e-02&--&            	3.0478e-02&--&            	2.6438e-01  & --    \\\hline
1/8&	6.6634e-02&      0.2473&	2.3927e-02&      0.3491&	2.6682e-01&     -0.0132\\\hline
1/16&4.3610e-02&      0.6116&	1.4726e-02&      0.7002&	2.2287e-01&      0.2597\\\hline
1//32&2.4691e-02&      0.8207&	8.2661e-03&      0.8331&	1.7450e-01&      0.3530\\\hline
 \end{tabular}
\end{center}
\end{table}

\begin{table}[h!]
\caption{Computational results for $\Delta t=h^2$ for the exact solution \eqref{solution}.}\label{tab:converge_h2}
\begin{center}\footnotesize
\renewcommand{\arraystretch}{1.3}
\begin{tabular}{|c||c|c||c|c||c|c|c|c||c|c||c|c||c|c||c|c||c|c|}\hline 
& \multicolumn{6}{c|}{errors for variable $u_1$}\\\hline
$h$&$L^\infty$ error &rate&$L^2$ error&rate&$H^1$ error&rate\\\hline
1/4&4.5933e-03&--&            	1.8486e-03&--&            	1.5748e-02&   --        	  	\\\hline
1/8&6.9876e-04&      2.7166&	2.8486e-04&      2.6981&	3.5475e-03&      2.1503\\\hline	
1/16&9.7215e-05 &     2.8456&	4.2196e-05&      2.7551&	8.0164e-04&      2.1458	\\\hline
& \multicolumn{6}{c|}{errors for variable $u_2$}\\\hline
$h$&$L^\infty$ error &rate&$L^2$ error&rate&$H^1$ error&rate\\\hline
1/4&5.7354e-03&--&            	2.5917e-03&--&            	1.8923e-02&   --      \\\hline
1/8&8.9519e-04&     2.6796&	4.1187e-04&      2.6537&	4.1429e-03&      2.1914	\\\hline
1/16&1.2319e-04&      2.8614&	6.0907e-05&      2.7575&	8.9762e-04&      2.2065	\\\hline
& \multicolumn{6}{c|}{errors for variable $p$}\\\hline
$h$&$L^\infty$ error &rate&$L^2$ error&rate&$H^1$ error&rate\\\hline
1/4&2.2096e-02&--&            	9.2527e-03&--&            	8.2484e-02&   --    \\\hline
1/8&8.8834e-03&      1.3146&	3.1754e-03&      1.5429&	4.1099e-02&      1.0050\\\hline
1/16&2.8736e-03&      1.6283&	9.4044e-04&      1.7555&	1.8666e-02&      1.1387\\\hline
 \end{tabular}
\end{center}
\end{table}

}

\appendix

\section{Proofs of theorems}

\subsection{Proof of Theorem \ref{th:ModularIMEXEuler}}\label{thm3.5proof}

{\allowdisplaybreaks

Take the inner product of \eqref{step1} with $2\Delta t w^{n+1}$ and integrate over $R^d$. By the polarization identity and the skewness of the nonlinear term, we have
\begin{equation}
\begin{aligned}
\Vert w^{n+1}\Vert_{L^2(\Omega)}^2-\Vert u^n \Vert_{L^2(\Omega)}^2& + \Vert w^{n+1}- u^n\Vert_{L^2(\Omega)}^2
+2\Delta t \nu \Vert \nabla w^{n+1}\Vert_{L^2(\Omega)} ^2 \label{Modu1}\\
& =2 \Delta t \int_{\Omega}f^{n+1} \cdot w^{n+1}dx -2\Delta t\gamma \int_{\Omega} (-\Delta)^\alpha
u^n \cdot w^{n+1} dx.
\end{aligned}
\end{equation}
Taking the inner product of \eqref{step2} with $u^{n+1}$ and integrating over $R^d$ yields
\begin{equation}\label{Modu2}
\begin{aligned}
\Delta t\gamma \frac{C_{d,\alpha}}{2}&\left(\vertiii{ u^{n+1}}^2_\alpha-\vertiii{ u^{n}}^2_\alpha+\vertiii {u^{n+1}-u^n}^2_\alpha\right)\\
&+\frac{1}{2} \Vert u^{n+1}\Vert^2_{L^2(\Omega)}+\frac{1}{2} \Vert w^{n+1}-u^{n+1}\Vert^2_{L^2(\Omega)}= \frac{1}{2} \Vert w^{n+1}\Vert^2_{L^2(\Omega)}.
\end{aligned}
\end{equation}
Similarly, taking inner product of \eqref{step2} with $(u^{n+1}+w^{n+1})/2$ and integrating over $R^d$ gives
\begin{equation} \label{Modu3}
\begin{aligned}
\Delta t\gamma \frac{C_{d,\alpha}}{4}&\left(\vertiii{ u^{n+1}}^2_\alpha-\vertiii{ u^{n}}^2_\alpha+\vertiii {u^{n+1}-u^n}^2_\alpha\right)\\
&+\Delta t\gamma \int_{\Omega} (-\Delta)^\alpha (u^{n+1}-u^n)\cdot w^{n+1}dx
+\frac{1}{2} \Vert u^{n+1}\Vert^2_{L^2(\Omega)}= \frac{1}{2} \Vert w^{n+1}\Vert^2_{L^2(\Omega)}.
\end{aligned}
\end{equation}
Adding \eqref{Modu1}, \eqref{Modu2}, and \eqref{Modu3} gives
\begin{equation}\label{Modu4}
\begin{aligned}
\Vert u^{n+1}\Vert_{L^2(\Omega)}^2-&\Vert u^n \Vert_{L^2(\Omega)}^2 + \Vert w^{n+1}- u^n\Vert_{L^2(\Omega)}^2
+2\Delta t \nu \Vert \nabla w^{n+1}\Vert_{L^2(\Omega)} ^2 \\
&+\Delta t\gamma \frac{3C_{d,\alpha}}{4}\left(\vertiii{ u^{n+1}}^2_\alpha-\vertiii{ u^{n}}^2_\alpha+\vertiii {u^{n+1}-u^n}^2_\alpha\right)\\
&+\Delta t\gamma \int_{\Omega} (-\Delta)^\alpha (u^{n+1}+u^n)\cdot w^{n+1} dx
 +\frac{1}{2} \Vert w^{n+1}-u^{n+1}\Vert^2_{L^2(\Omega)}\\
 &=2 \Delta t \int_{\Omega}f^{n+1} w^{n+1}dx.
\end{aligned}
\end{equation}
Now taking inner product of \eqref{step2} with $\Delta t\gamma(-\Delta)^\alpha (u^{n+1}+u^n)$ and integrating over $\Omega$, we obtain
$$
\begin{aligned}
2\Delta  t^2\gamma^2&\left(\Vert (-\Delta)^\alpha u^{n+1}\Vert_{L^2(\Omega)}^2-\Vert (-\Delta)^\alpha u^n\Vert_{L^2(\Omega)}^2\right)\\
&+\Delta t\gamma \frac{C_{d,\alpha}}{4}\left(\vertiii{ u^{n+1}}^2_\alpha-\vertiii{ u^{n}}^2_\alpha+\vertiii {u^{n+1}+u^n}^2_\alpha\right)\\
&=\Delta t\gamma \int_{\Omega} (-\Delta)^\alpha (u^{n+1}+u^n) \cdot w^{n+1}dx.
\end{aligned}
$$
Substituting the above equation into \eqref{Modu4} and applying Young's inequality to the right-hand side, we have
\begin{equation}\label{Modu5}
\begin{aligned}
\Vert u^{n+1}\Vert_{L^2(\Omega)}^2-&\Vert u^n \Vert_{L^2(\Omega)}^2 + \Vert w^{n+1}- u^n\Vert_{L^2(\Omega)}^2
+\Delta t \nu \Vert \nabla w^{n+1}\Vert_{L^2(\Omega)} ^2 \\
&+\Delta t\gamma C_{d,\alpha}\left(\vertiii{ u^{n+1}}^2_\alpha-\vertiii{ u^{n}}^2_\alpha\right)+\Delta t\gamma \frac{3C_{d,\alpha}}{4}\vertiii {u^{n+1}-u^n}^2_\alpha\\
&
+2\Delta  t^2\gamma^2\left(\Vert (-\Delta)^\alpha u^{n+1} \Vert_{L^2(\Omega)}^2-\Vert (-\Delta)^\alpha u^n\Vert_{L^2(\Omega)}^2\right) \\
&+\Delta t\gamma \frac{C_{d,\alpha}}{4}\vertiii {u^{n+1}+u^n}^2_\alpha+\frac{1}{2} \Vert w^{n+1}-u^{n+1}\Vert^2_{L^2(\Omega)}\\
&\leq \frac{\Delta t}{\nu}\Vert f^{n+1} \Vert_{-1}^2.
\end{aligned}
\end{equation}
Summing \eqref{Modu5} from $n=0$ to $N-1$ results in \eqref{thmineq:ModularFracBE}.

}

\subsection{Proof of Theorem \ref{thm:conv+Modu}}\label{thm3.6proof}

{ \allowdisplaybreaks

The exact solution ${u}$ of the \eqref{eq:model} satisfies
\begin{equation}\label{eq:convtrue2}
\begin{aligned}
\frac{{u}(t_{n+1})-{u}(t_{n})}{\Delta t}& +({u}(t_{n+1})\cdot \nabla)
{u}(t_{n+1})-\nu\Delta {u}(t_{n+1}) +\gamma (-\Delta)^\alpha {u}(t_{n+1})\\
&+\nabla {p}(t_{n+1})={f}(t_{n+1}) + R({u}(t_{n+1})),
\end{aligned}
\end{equation}
where $R({u}(t_{n+1}))$ is defined as
$$
R({u}(t_{n+1}))=\frac{{u}(t_{n+1})-{u}(t_{n})}{\Delta t}-{u}_{t}
(t_{n+1})\text{ .}
$$
Subtracting (\ref{step1})  from (\ref{eq:convtrue2}) gives
\begin{equation}\label{err}
\begin{aligned}
\frac{\bar{e}^{n+1}-e^n}{\Delta t}+&({u}(t_{n+1})\cdot\nabla){u}(t_{n+1})
-(u^{n}\cdot\nabla)w^{n+1}
\\
& -\nu \Delta \bar{e}^{n+1}+\gamma (-\Delta)^{\alpha} e^n+\gamma (-\Delta)^{\alpha} (u(t_{n+1})-u(t_n))
\\
&\qquad+\nabla {p}(t_{n+1})-\nabla p^{n+1}=  R({u}(t_{n+1})) \quad \text{ in } \Omega. 
\end{aligned}
\end{equation}
Taking inner product of \eqref{err} with $2\Delta t\bar{e}^{n+1}$,  integrating over $R^d$, and using the polarization identity gives
\begin{equation}\label{err11}
\begin{aligned}
\Vert \bar{e}^{n+1}\Vert^2_{L^2(\Omega)}-&\Vert e^n\Vert^2_{L^2(\Omega)}+\Vert \bar{e}^{n+1}-e^n\Vert^2_{L^2(\Omega)}+2\Delta t\nu \Vert \nabla \bar{e}^{n+1}\Vert^2_{L^2(\Omega)}
\\
&+2\Delta t\int_{\Omega}(\nabla {p}(t_{n+1})-\nabla p^{n+1})\cdot \bar{e}^{n+1}dx +2\Delta t\gamma \int_{\Omega } (-\Delta)^{\alpha} e^n\cdot \bar{e}^{n+1}dx\\
 &=-2\Delta t {b_\Omega}({u}(t_{n+1}), {u}(t_{n+1}),\bar{e}^{n+1})+2\Delta t {b_\Omega}(u^{n},w^{n+1},\bar{e}^{n+1})\\
 &\quad-2\Delta t\gamma \int_{\Omega}(-\Delta)^{\alpha} (u(t_{n+1})-u(t_n)) \cdot \bar{e}^{n+1}dx\\
 &\quad+ 2 \Delta t\int_{\Omega}R({u}(t_{n+1}))\cdot \bar{e}^{n+1}dx.
\end{aligned}
\end{equation}
Because $\nabla \cdot u(t_{n+1})=0$ and $\nabla\cdot w^{n+1}=0$, we have $\nabla \cdot \bar{e}^{n+1}=0$ and thus
$$
\begin{aligned}
2\Delta t\int_{\Omega}(\nabla {p}(t_{n+1})-\nabla p^{n+1})\cdot \bar{e}^{n+1}dx=2\Delta t\int_{\Omega}( {p}(t_{n+1})-p^{n+1})\cdot (\nabla\cdot\bar{e}^{n+1})dx=0.
\end{aligned}
$$
\eqref{step2} can be rewritten as
\begin{equation}\label{err:step3}
\begin{aligned}
2\Delta t\gamma(-\Delta )^{\alpha} (e^{n+1}-e^n)-&2\Delta t\gamma(-\Delta )^{\alpha} ({u}(t_{n+1})-{u}(t_{n}))
+ e^{n+1}-\bar{e}^{n+1}  =0 \,\,\, \text{in } \Omega.
\end{aligned}
\end{equation}
Taking the inner product of \eqref{err:step3} with $e^{n+1}$ and integrating over $R^d$ gives
\begin{equation}\label{err22}
\begin{aligned}
\Delta t\gamma \frac{C_{d,\alpha}}{2}&\left(\vertiii{e^{n+1}}_{\alpha}^2-\vertiii{e^{n}}_{\alpha}^2+\vertiii{e^{n+1}-e^n}_{\alpha}^2\right)\\
&-2\Delta t\gamma \int_{\Omega}(-\Delta)^{\alpha} ({u}(t_{n+1})-{u}(t_{n}))\cdot e^{n+1}dx\\
&+\frac{1}{2}\Vert e^{n+1}\Vert^2_{L^2(\Omega)} +\frac{1}{2}\Vert e^{n+1}-\bar{e}^{n+1}\Vert^2_{L^2(\Omega)}
=\frac{1}{2}\Vert \bar{e}^{n+1}\Vert^2_{L^2(\Omega)}. 
\end{aligned}
\end{equation}
Taking the inner product of \eqref{err:step3} with $(e^{n+1}+\bar{e}^{n+1})/2$ and integrating over $R^d$, we obtain
\begin{equation}\label{err33}
\begin{aligned}
\Delta t\gamma \frac{C_{d,\alpha}}{4}&\left(\vertiii{e^{n+1}}_{\alpha}^2
-\vertiii{e^{n}}_{\alpha}^2+\vertiii{e^{n+1}-e^n}_{\alpha}^2\right)\\
&+\Delta t\gamma \int_{\Omega}(-\Delta)^{\alpha} (e^{n+1}-e^n)\cdot \bar{e}^{n+1}dx\\
&-\Delta t\gamma \int_{\Omega}(-\Delta)^{\alpha} ({u}(t_{n+1})-{u}(t_{n}))\cdot (e^{n+1}+\bar{e}^{n+1}) dx\\
&+\frac{1}{2}\Vert e^{n+1}\Vert^2_{L^2(\Omega)} =\frac{1}{2}\Vert \bar{e}^{n+1}\Vert^2_{L^2(\Omega)}. 
\end{aligned}
\end{equation}
Adding \eqref{err11}, \eqref{err22}, and \eqref{err33} gives
\begin{equation}\label{err44}
\begin{aligned}
\Vert e^{n+1}\Vert^2_{L^2(\Omega)}-&\Vert e^n\Vert^2_{L^2(\Omega)}+\Vert \bar{e}^{n+1}-e^n\Vert^2_{L^2(\Omega)}+\frac{1}{2}\Vert e^{n+1}-\bar{e}^{n+1}\Vert^2_{L^2(\Omega)}\\
&
+\Delta t\gamma \frac{3C_{d,\alpha}}{4}\left(\vertiii{e^{n+1}}_{\alpha}^2
-\vertiii{e^{n}}_{\alpha}^2+\vertiii{e^{n+1}-e^n}_{\alpha}^2\right)\\
& +\Delta t\gamma \int_{\Omega}(-\Delta)^{\alpha} (e^{n+1}+e^n)\cdot\bar{e}^{n+1}+2\Delta t\nu \Vert \nabla \bar{e}^{n+1}\Vert^2_{L^2(\Omega)}\\
& -\Delta t\gamma \int_{\Omega}(-\Delta)^{\alpha} ({u}(t_{n+1})-{u}(t_{n})) (3e^{n+1}-\bar{e}^{n+1}) dx\\
& =-2\Delta t {b_\Omega}({u}(t_{n+1}), {u}(t_{n+1}),\bar{e}^{n+1})+2\Delta t {b_\Omega}({u}^n,w^{n+1},\bar{e}^{n+1})\\
& \quad+ 2\Delta t\int_{\Omega}R({u}(t_{n+1}))\cdot \bar{e}^{n+1}dx.
\end{aligned}
\end{equation}
Taking inner product of  \eqref{err:step3} with $\Delta t\gamma(-\Delta)^\alpha (e^{n+1}+e^n)$ and integrating over $\Omega$ gives
\begin{equation}\label{err55}
\begin{aligned}
2\Delta  t^2\gamma^2&\left(\Vert (-\Delta)^\alpha e^{n+1}\Vert_{L^2(\Omega)}^2-\Vert (-\Delta)^\alpha e^n\Vert_{L^2(\Omega)}^2\right)\\
&-2\Delta t^2\gamma^2 \int_{\Omega}(-\Delta)^\alpha ({u}(t_{n+1})-{u}(t_{n})) \cdot(-\Delta)^\alpha (e^{n+1}+e^n)dx\\
&+\Delta t\gamma \frac{C_{d,\alpha}}{4}\left(\vertiii{ e^{n+1}}^2_{\alpha}-\vertiii{ e^{n}}^2_{\alpha}+\vertiii {e^{n+1}+e^n}^2_{\alpha}\right)\\
&=\Delta t\gamma \int_{\Omega  } (-\Delta)^{\alpha}(e^{n+1}+e^n)\cdot \bar{e}^{n+1}dx.
\end{aligned}
\end{equation}
Adding \eqref{err44} and \eqref{err55} gives
\begin{equation}\label{err66}
\begin{aligned}
\Vert e^{n+1}\Vert^2_{L^2(\Omega)}-&\Vert e^n\Vert^2_{L^2(\Omega)}+\Vert \bar{e}^{n+1}-e^n\Vert^2_{L^2(\Omega)}+\frac{1}{2}\Vert e^{n+1}
-\bar{e}^{n+1}\Vert^2_{L^2(\Omega)}\\
&+2\Delta t\nu \Vert \nabla \bar{e}^{n+1}\Vert^2_{L^2(\Omega)}
+\Delta t\gamma C_{d,\alpha}\left(\vertiii{e^{n+1}}_\alpha^2
-\vertiii{e^{n}}_\alpha^2\right)\\
&+\Delta t\gamma \frac{3C_{d,\alpha}}{4}\vertiii{e^{n+1}-e^n}_\alpha^2
+\Delta t\gamma \frac{C_{d,\alpha}}{4}\vertiii {e^{n+1}+e^n}^2_\alpha\\
&+2\Delta  t^2\gamma^2\left(\Vert (-\Delta)^\alpha e^{n+1}\Vert_{L^2(\Omega)}^2-\Vert (-\Delta)^\alpha e^n\Vert_{L^2(\Omega)}^2\right)\\
& =-2\Delta t {b_\Omega}({u}(t_{n+1}), {u}(t_{n+1}),\bar{e}^{n+1})+2\Delta t {b_\Omega}(u^{n},w^{n+1},\bar{e}^{n+1})\\
&+2\Delta t^2\gamma^2 \int_{\Omega}(-\Delta)^\alpha ({u}(t_{n+1})-{u}(t_{n})) \cdot(-\Delta)^\alpha (e^{n+1}+e^n)dx\\
&  +\Delta t\gamma \int_{\Omega}(-\Delta)^\alpha ({u}(t_{n+1})-{u}(t_{n}))\cdot (3e^{n+1}-\bar{e}^{n+1}) dx\\
&+ 2\Delta t\int_{\Omega}R({u}(t_{n+1}))\cdot\bar{e}^{n+1}dx.
\end{aligned}
\end{equation}
\noindent The nonlinear terms can be rewritten as 
$$%\begin{equation}
\begin{aligned}
 {b_\Omega}(u^{n}&,w^{n+1},\bar{e}^{n+1})
-{b_\Omega}({u}(t_{n+1}), {u}(t_{n+1}),\bar{e}^{n+1})\\
&={b_\Omega}(e^{n}, \bar{e}^{n+1},\bar{e}^{n+1})-{b_\Omega}(e^{n}
,{u}(t_{n+1}),\bar{e}^{n+1})\\
&\quad-{b_\Omega}({u}(t_{n}),\bar{e}^{n+1},\bar{e}^{n+1})-{b_\Omega}(({u}(t_{n+1})-{u}(t_{n})),{u}(t_{n+1}),\bar{e}^{n+1})\\
&=-{b_\Omega}(e^{n}
,{u}(t_{n+1}),\bar{e}^{n+1})-{b_\Omega}(({u}(t_{n+1})-{u}(t_{n})),{u}(t_{n+1}),\bar{e}^{n+1}).
\end{aligned}
$$%\end{equation}
Using Young's inequality and ${u}\in L^{\infty}(0,T;H^{2}(\Omega))$, we have
$$%\begin{equation}
\begin{aligned}
2\Delta t\vert {b_\Omega}(&e^{n},{u}(t_{n+1}),\bar{e}^{n+1}) \vert\\
&\leq C\Delta t\|
e^{n}\|_{L^2(\Omega)}\|{u}(t_{n+1})\|_{H^2(\Omega)}\|\nabla \bar{e}^{n+1}\|_{L^2(\Omega)}\\
&\leq \frac{\nu}{16}\Delta t\|\nabla \bar{e}^{n+1}\|^{2}_{L^2(\Omega)}+C\nu^{-1}\Delta t\|e^{n}\|^{2}_{L^2(\Omega)}
\end{aligned}
$$%\end{equation}
and
$$
\begin{aligned}
2\Delta t \vert {b_\Omega}(&({u}(t_{n+1})-{u}(t_{n})),{u}(t_{n+1}),\bar{e}^{n+1})\vert\\
&\leq C\Delta t\|  ({u}(t_{n+1})-{u}(t_{n}))\|_{L^2(\Omega)}\Vert  {u}(t_{n+1})\Vert_{H^2(\Omega)} \Vert\nabla \bar{e}^{n+1}\Vert_{L^2(\Omega)}\\
&\leq \frac{\nu}{16}\Delta t\Vert \nabla \bar{e}^{n+1}\Vert^2_{L^2(\Omega)}+C \nu^{-1}\Delta t\Vert  ({u}(t_{n+1})
-{u}(t_{n}))\Vert^2_{L^2(\Omega)}\\
&\leq \frac{\nu}{16}\Delta t\Vert \nabla \bar{e}^{n+1}\Vert^2_{L^2(\Omega)}+C \nu^{-1}\Delta t ^2\int_{t_n}^{t_{n+1}}\Vert {u}_t\Vert^2_{L^2(\Omega)}dt.
\nonumber
\end{aligned}
$$
Next, 
$$
\begin{aligned}
2\Delta t\int_{\Omega}&
R({u}(t_{n+1}))\cdot\bar{e}^{n+1}dx\\
&=2\Delta t\int_{\Omega}\left(\frac{{u}(t_{n+1})
-{u}(t_{n})}{\Delta
t}-{u}_{t}(t_{n+1})\right)\cdot \bar{e}^{n+1}dx\\
&\leq C\Delta t\|\frac{{u}(t_{n+1})-{u}(t_{n})}{\Delta t}-{u}_{t}(t_{n+1})\| _{L^2(\Omega)}\|\nabla
\bar{e}^{n+1}\|_{L^2(\Omega)}\\
&\leq\frac{\nu}{16}\Delta t\|\nabla \bar{e}^{n+1}\|^{2}_{L^2(\Omega)}+\frac{C}{\nu}\Delta t\|\frac
{{u}(t_{n+1})-{u}(t_{n})}{\Delta t}-{u}_{t}(t_{n+1})\|^{2}_{L^2(\Omega)}\\
&\leq\frac{\nu}{16}\Delta t\|\nabla \bar{e}^{n+1}\|^{2}_{L^2(\Omega)}+\frac{C\Delta t^2}{\nu}
\int_{t_{n}}^{t_{n+1}}\|{u}_{tt}\|^{2}_{L^2(\Omega)} dt .
\end{aligned}
$$
The last two terms of the right-hand side of \eqref{err66} are bounded as follows:
$$
\begin{aligned}
 2\gamma^2\Delta t^2 \int_{\Omega}&(-\Delta)^\alpha(u(t_{n+1})-u(t_n)) \cdot(-\Delta)^\alpha (e^{n+1}+e^n)dx\\
&\leq 2\gamma^2\Delta t^2 \frac{C_{d,\alpha}}{2}\vertiii{(-\Delta )^\alpha (u(t_{n+1})
-u(t_n))}_\alpha\vertiii{e^{n+1}+e^n}_\alpha\\
&\leq 4\gamma^3\Delta t^3 \frac{C_{d,\alpha}}{2}\vertiii{(-\Delta )^\alpha (u(t_{n+1})-u(t_n))}_\alpha^2+\Delta t \gamma \frac{C_{d,\alpha}}{4}\vertiii{e^{n+1}+e^n}_\alpha^2\\
&\leq 4\gamma^4\Delta t^4 \frac{C_{d,\alpha}}{2}\int_{t_n}^{t_{n+1}}\vertiii{(-\Delta )^\alpha u_t}_\alpha^2dt+\Delta t \gamma \frac{C_{d,\alpha}}{4}\vertiii{e^{n+1}+e^n}_\alpha^2
\end{aligned}
$$
and
$$
\begin{aligned}
\gamma\Delta t \int_{\Omega}&(-\Delta)^\alpha ({u}(t_{n+1})-{u}(t_{n}))\cdot (3e^{n+1}-\bar{e}^{n+1}) dx\\
&=3\gamma\Delta t \int_{\Omega}(-\Delta)^\alpha ({u}(t_{n+1})-{u}(t_{n}))\cdot (e^{n+1}-\bar{e}^{n+1}) dx\\
&\quad+2\gamma\Delta t \int_{\Omega}(-\Delta)^\alpha ({u}(t_{n+1})-{u}(t_{n}))\cdot \bar{e}^{n+1} dx\\
&\leq 3\gamma\Delta t \Vert (-\Delta)^\alpha ({u}(t_{n+1})-{u}(t_{n})) \Vert_{L^2(\Omega)}\Vert e^{n+1} -\bar{e}^{n+1}\Vert_{L^2(\Omega)}\\
&\quad+2\gamma\Delta t \Vert (-\Delta)^\alpha ({u}(t_{n+1})-{u}(t_{n}))\Vert_{L^2(\Omega)}\Vert \bar{e}^{n+1}\Vert_{L^2(\Omega)}\\
&\leq \frac{9}{2}\gamma^2\Delta t^2 \Vert (-\Delta)^\alpha ({u}(t_{n+1})-{u}(t_{n})) \Vert_{L^2(\Omega)}^2+\frac{1}{2}\Vert e^{n+1} -\bar{e}^{n+1}\Vert_{L^2(\Omega)}^2\\
&\quad+C\gamma^2\Delta t\nu^{-1}\Vert (-\Delta)^\alpha ({u}(t_{n+1})-{u}(t_{n}))\Vert_{L^2(\Omega)}^2+\Delta t \nu \Vert \nabla \bar{e}^{n+1}\Vert^2_{L^2(\Omega)}\\
&\leq \frac{9}{2}\gamma^2\Delta t^3\int_{t_n}^{t_{n+1}} \Vert (-\Delta)^\alpha {u}_t \Vert_{L^2(\Omega)}^2dt+\frac{1}{2}\Vert e^{n+1} -\bar{e}^{n+1}\Vert_{L^2(\Omega)}^2\\
&\quad+C\gamma^2\Delta t^2\nu^{-1}\int_{t_n}^{t_{n+1}}\Vert (-\Delta)^\alpha {u}_t\Vert_{L^2(\Omega)}^2dt+\Delta t \nu \Vert \nabla \bar{e}^{n+1}\Vert^2_{L^2(\Omega)}
\end{aligned}
$$
Combining all the estimates above, we now have
\begin{equation}\label{err77}
\begin{aligned}
\Vert e^{n+1}\Vert^2_{L^2(\Omega)}-&\Vert e^n\Vert^2_{L^2(\Omega)}+\Vert \bar{e}^{n+1}-e^n\Vert^2_{L^2(\Omega)}
+\frac{1}{2}\Vert e^{n+1}-\bar{e}^{n+1}\Vert^2_{L^2(\Omega)}\\
&+\Delta t\nu \Vert \nabla \bar{e}^{n+1}\Vert^2_{L^2(\Omega)}
+\Delta t\gamma C_{d,\alpha}\left(\vertiii{e^{n+1}}_\alpha^2-\vertiii{e^{n}}_\alpha^2\right)\\
&+\Delta t\gamma \frac{3C_{d,\alpha}}{4}\vertiii{e^{n+1}-e^n}_\alpha^2
+\Delta t\gamma \frac{C_{d,\alpha}}{4}\vertiii {e^{n+1}+e^n}^2_\alpha\\
&+2\Delta  t^2\gamma^2\left(\Vert (-\Delta)^\alpha e^{n+1}\Vert_{L^2(\Omega)}^2-\Vert (-\Delta)^\alpha e^n\Vert_{L^2(\Omega)}^2\right)\\
& \leq C\nu^{-1}\Delta t\|e^{n}\|^{2}_{L^2(\Omega)}+C \nu^{-1}\Delta t ^2\int_{t_n}^{t_{n+1}}\Vert {u}_t\Vert^2_{L^2(\Omega)}dt\\
 &+\frac{C\Delta t^2}{\nu}
\int_{t_{n}}^{t_{n+1}}\|{u}_{tt}\|^{2}_{L^2(\Omega)} dt
+4\Delta t^4\gamma^3 \frac{C_{d,\alpha}}{2}\int_{t_n}^{t_{n+1}}\vertiii{(-\Delta )^\alpha {u}_t}_\alpha^2dt\\
&+\frac{9}{2}\Delta t^3\gamma^2\int_{t_n}^{t_{n+1}} \Vert (-\Delta)^\alpha {u}_t \Vert_{L^2(\Omega)}^2dt\\
& + C\Delta t^2\gamma^2\nu^{-1}\int_{t_n}^{t_{n+1}}\Vert (-\Delta)^\alpha {u}_t\Vert_{L^2(\Omega)}^2dt.
\end{aligned}
\end{equation}
Taking the sum of (\ref{err77}) from $n=0$ to
$n=N-1$ gives
$$
\begin{aligned}
\Vert e^{N}&\Vert^2_{L^2(\Omega)}+\sum_{n=0}^{N-1}\Vert \bar{e}^{n+1}-e^n\Vert^2_{L^2(\Omega)}+\sum_{n=0}^{N-1}\Vert e^{n+1}-\bar{e}^n\Vert^2_{L^2(\Omega)}\\
&+\Delta t\sum_{n=0}^{N-1}\nu \Vert \nabla \bar{e}^{n+1}\Vert^2_{L^2(\Omega)}
+\Delta t\gamma C_{d,\alpha}\vertiii{e^N}_\alpha^2+\Delta t\sum_{n=0}^{N-1}\gamma \frac{3C_{d,\alpha}}{4}\vertiii{e^{n+1}-e^n}_\alpha^2\\
&+\Delta t\sum_{n=0}^{N-1}\gamma \frac{C_{d,\alpha}}{4}\vertiii{e^{n+1}+e^n}_\alpha^2
+2\Delta  t^2\gamma^2\Vert (-\Delta)^\alpha e^{N}\Vert_{L^2(\Omega)}^2\\
 &\leq \Vert e^0\Vert^2_{L^2(\Omega)}+\Delta t\gamma
 C_{d,\alpha}\vertiii{e^0}_\alpha^2+2\Delta  t^2\gamma^2\Vert (-\Delta)^\alpha e^{0}\Vert_{L^2(\Omega)}^2\\
 &\quad+C\nu^{-1}\Delta t\sum_{n=0}^{N-1}\|e^{n}\|^{2}_{L^2(\Omega)}+C \nu^{-1}\Delta t ^2\int_{0}^{T}\Vert {u}_t\Vert^2_{L^2(\Omega)}dt\\
 &\quad+\frac{C\Delta t^2}{\nu}
\int_{0}^{T}\|{u}_{tt}\|^{2}_{L^2(\Omega)} dt+4\Delta t^4\gamma^3 \frac{C_{d,\alpha}}{2}\int_{0}^{T}\vertiii{(-\Delta )^\alpha {u}_t}_\alpha^2dt\\
&\quad+\frac{9}{2}\Delta t^3\gamma^2\int_{0}^{T} \Vert (-\Delta)^\alpha {u}_t \Vert_{L^2(\Omega)}^2dt\\
& \quad+ C\Delta t^2\gamma^2\nu^{-1}\int_{0}^{T}\Vert (-\Delta)^\alpha {u}_t\Vert_{L^2(\Omega)}^2dt.
\end{aligned}
$$
As ${u}(t_0)=u^0=w^0$, we have $e^0=0$. Applying the discrete Gronwall
inequality \cite[p. 176]{GR79}, we have
$$
\begin{aligned}
\Vert e^{N}&\Vert^2_{L^2(\Omega)}+\sum_{n=0}^{N-1}\Vert \bar{e}^{n+1}-e^n\Vert^2_{L^2(\Omega)}+\sum_{n=0}^{N-1}\Vert e^{n+1}-\bar{e}^n\Vert^2_{L^2(\Omega)}\\
&+\Delta t\sum_{n=0}^{N-1}\nu \Vert \nabla \bar{e}^{n+1}\Vert^2_{L^2(\Omega)}+\Delta t\gamma C_{d,\alpha}\vertiii{e^N}_\alpha^2
+\Delta t\sum_{n=0}^{N-1}\gamma \frac{3C_{d,\alpha}}{4}\vertiii{e^{n+1}-e^n}_\alpha^2\\
&+\Delta t\sum_{n=0}^{N-1}\gamma \frac{C_{d,\alpha}}{4}\vertiii{e^{n+1}+e^n}_\alpha^2+2\Delta  t^2\gamma^2\Vert (-\Delta)^\alpha e^{N}\Vert_{L^2(\Omega)}^2\\
& \leq exp\left(\frac{CT}{\nu}\right)\bigg\{C \nu^{-1}\Delta t ^2\int_{0}^{T}\Vert {u}_t\Vert^2_{L^2(\Omega)}dt
 +\frac{C\Delta t^2}{\nu}
\int_{0}^{T}\|{u}_{tt}\|^{2}_{L^2(\Omega)} dt\\
&\quad+4\Delta t^4\gamma^3 \frac{C_{d,\alpha}}{2}\int_{0}^{T}\vertiii{(-\Delta )^\alpha {u}_t}_\alpha^2dt
+\frac{9}{2}\Delta t^3\gamma^2\int_{0}^{T} \Vert (-\Delta)^\alpha {u}_t \Vert_{L^2(\Omega)}^2dt\\
& \quad+ C\Delta t^2\gamma^2\nu^{-1}\int_{0}^{T}\Vert (-\Delta)^\alpha {u}_t\Vert_{L^2(\Omega)}^2dt\bigg\}.
\end{aligned}
$$
Because ${u}\in  H^{2}(0,T;L^{2}(\Omega))$ and $(-\Delta)^\alpha{u}\in  H^{1}(0,T;H^{\alpha}(R^d))$, assuming $\Delta t\leq 1$, after absorbing constants, we have 
\begin{equation}\label{err00}
\begin{aligned}
\Vert e^{N}\Vert^2_{L^2(\Omega)}+&\sum_{n=1}^{N-1}\Vert \bar{e}^{n+1}-e^n\Vert^2_{L^2(\Omega)}+\sum_{n=1}^{N-1}\Vert e^{n+1}-\bar{e}^n\Vert^2_{L^2(\Omega)}\\
&+\Delta t\sum_{n=1}^{N-1}\nu \Vert \nabla \bar{e}^{n+1}\Vert^2_{L^2(\Omega)}+\Delta t\gamma C_{d,\alpha}\vertiii{e^N}_\alpha^2\\
&+\Delta t\sum_{n=1}^{N-1}\gamma \frac{3C_{d,\alpha}}{4}\vertiii{e^{n+1}-e^n}_\alpha^2+\Delta t\sum_{n=1}^{N-1}\gamma \frac{C_{d,\alpha}}{4}\vertiii{e^{n+1}+e^n}_\alpha^2\\
&+2\Delta  t^2\gamma^2\Vert (-\Delta)^\alpha e^{N}\Vert_{L^2(\Omega)}^2\\
& \leq C \left(\Delta t^2+\Delta t^3+\Delta t^4\right)\leq C\Delta t^2
\end{aligned}
\end{equation}
which completes the proof.

}

\subsection{Proof of Theorem \ref{thm:trancate}}\label{thm4.1proof}

{ \allowdisplaybreaks

The true solution $(u, p)$ of \eqref{eq:model} satisfies
\begin{equation}\label{true}
\begin{aligned}
 \int_{\Omega }u_t v dx +& {b_\Omega} (u, u, v)+\nu \int_{\Omega} \nabla u : \nabla v dx \\
 &+\gamma \frac{C_{d,\alpha}}{2}\int_{R^d }\int_{R^d} \frac{u(x)-u(y)}{|x-y|^{d+2\alpha}} (v(x)-v(y))dydx= \int_{\Omega} f v dx \quad\forall\, v\in V
\end{aligned}
\end{equation}
\eqref{prob} 
is equivalent to
\begin{equation}\label{trancate}
\begin{aligned}
 \int_{\Omega }&\frac{\partial{u}_\lambda}{\partial t} v dx +{b_\Omega} ({u}_\lambda, {u}_\lambda, v)+\nu \int_{\Omega} \nabla {u}_\lambda : \nabla v dx \\
&\qquad\qquad+ \gamma \frac{C_{d,\alpha}}{2}\int_{\Omega \cup \Omega_\lambda }\int_{(\Omega\cup\Omega_\lambda)\cap\mathcal{B}_{\lambda}(x)} \frac{{u}_\lambda(x)-{u}_\lambda(y)}{|x-y|^{d+2\alpha}} (v(x)-v(y))dydx\\&\qquad= \int_{\Omega} f v dx \quad\forall\, v\in V
\end{aligned}
\end{equation}
Let $e=u-{u}_\lambda$. Subtracting \eqref{trancate} from \eqref{true} gives
\begin{equation}\label{trancate_error}
\left\{\begin{aligned}
& \int_{\Omega }e_t v dx +{b_\Omega} (u, u, v)-{b_\Omega} ({u}_\lambda, {u}_\lambda, v)+\nu \int_{\Omega} \nabla e : \nabla v dx\\
&\quad\quad +\gamma \frac{C_{d,\alpha}}{2}\int_{R^d }\int_{R^d} \frac{u(x)-u(y)}{|x-y|^{d+2\alpha}}(v(x)-v(y))dydx\\
& -\gamma \frac{C_{d,\alpha}}{2}\int_{\Omega \cup \Omega_\lambda }\int_{(\Omega\cup\Omega_\lambda)\cap\mathcal{B}_{\lambda}(x)}  \frac{{u}_\lambda(x)-{u}_\lambda(y)}{|x-y|^{d+2\alpha}}(v(x)-v(y))dydx= 0 \quad\forall\, v\in V
\end{aligned}\right.
\end{equation}
The nonlinear terms can be rewritten as
\begin{align}
&{b_\Omega}(u,u,v)-{b_\Omega}({u}_\lambda,{u}_\lambda,v)\\
&={b_\Omega}(u,u,v)-{b_\Omega}({u}_\lambda,u,v)+{b_\Omega}({u}_\lambda,u,v)-{b_\Omega}({u}_\lambda,{u}_\lambda,v)\nonumber\\
&={b_\Omega}(e,u,v)+{b_\Omega}({u}_\lambda,e,v)\nonumber
\end{align}
Let $\Omega_c=R^d \setminus (\Omega\cup \Omega_{\lambda})$. We rewrite the nonlocal terms as follows
\begin{align}
&\gamma \frac{C_{d,\alpha}}{2}\int_{R^d }\int_{R^d} \frac{u(x)-u(y)}{|x-y|^{d+2\alpha}}(v(x)-v(y))dydx\\
& \qquad\qquad-\gamma \frac{C_{d,\alpha}}{2}\int_{\Omega \cup \Omega_\lambda }\int_{(\Omega\cup\Omega_\lambda)\cap\mathcal{B}_{\lambda}(x)}  \frac{{u}_\lambda(x)-{u}_\lambda(y)}{|x-y|^{d+2\alpha}}(v(x)-v(y))dydx\nonumber\\
&=\int_{\Omega\cup\Omega_{\lambda}}
\int_{\Omega_c}\frac{u(x)-u(y)}{|x-y|^{d+2\alpha}}(v(x)-v(y))dydx+\int_{\Omega_c}\int_{R^d}\frac{u(x)-u(y)}{|x-y|^{d+2\alpha}}(v(x)-v(y))dydx\nonumber\\
&\qquad+\int_{\Omega\cup\Omega_{\lambda}}
\int_{(\Omega\cup\Omega_{\lambda})\cap B_{\lambda}(x)}\frac{u(x)-u(y)}{|x-y|^{d+2\alpha}}(v(x)-v(y))dydx\nonumber\\
&\qquad+\int_{\Omega\cup\Omega_{\lambda}}
\int_{(\Omega\cup\Omega_{\lambda})\setminus B_{\lambda}(x)}\frac{u(x)-u(y)}{|x-y|^{d+2\alpha}}(v(x)-v(y))dydx\nonumber\\
&\qquad-\int_{\Omega\cup\Omega_{\lambda}}
\int_{(\Omega\cup\Omega_{\lambda})\cap B_{\lambda}(x)}\frac{{u}_\lambda(x)-{u}_\lambda(y)}{|x-y|^{d+2\alpha}}(v(x)-v(y))dydx\nonumber\\
&=\int_{\Omega\cup\Omega_{\lambda}}
\int_{(\Omega\cup\Omega_{\lambda})\cap B_{\lambda}(x)}\frac{e(x)-e(y)}{|x-y|^{d+2\alpha}}(v(x)-v(y))dydx\nonumber\\
&\quad+\int_{\Omega\cup\Omega_{\lambda}}
\int_{\Omega_c}\frac{u(x)-u(y)}{|x-y|^{d+2\alpha}}(v(x)-v(y))dydx+\int_{\Omega_c}\int_{R^d}\frac{u(x)-u(y)}{|x-y|^{d+2\alpha}}(v(x)-v(y))dydx\nonumber\\
&\qquad+\int_{\Omega\cup\Omega_{\lambda}}
\int_{(\Omega\cup\Omega_{\lambda})\setminus B_{\lambda}(x)}\frac{u(x)-u(y)}{|x-y|^{d+2\alpha}}(v(x)-v(y))dydx\nonumber\\
&=\int_{\Omega\cup\Omega_{\lambda}}
\int_{(\Omega\cup\Omega_{\lambda})\cap B_{\lambda}(x)}\frac{e(x)-e(y)}{|x-y|^{d+2\alpha}}(v(x)-v(y))dydx\nonumber\\
&\qquad+\int_{\Omega\cup\Omega_{\lambda}}
\int_{R^d\setminus B_{\lambda}(x)}\frac{u(x)-u(y)}{|x-y|^{d+2\alpha}}(v(x)-v(y))dydx\nonumber\\
&\qquad+\int_{\Omega_c}\int_{R^d}\frac{u(x)-u(y)}{|x-y|^{d+2\alpha}}(v(x)-v(y))dydx\nonumber\\
\end{align}
%\begin{align}
%&\gamma \frac{C_{d,\alpha}}{2}\int_{R^d }\int_{R^d} \frac{u(x)-u(y)}{|x-y|^{d+2\alpha}}(v(x)-v(y))dydx\\
%& \qquad\qquad-\gamma \frac{C_{d,\alpha}}{2}\int_{\Omega \cup \Omega_\lambda }\int_{(\Omega\cup\Omega_\lambda)\cap\mathcal{B}_{\lambda}(x)}  \frac{{u}_\lambda(x)-{u}_\lambda(y)}{|x-y|^{d+2\alpha}}(v(x)-v(y))dydx\nonumber\\
%&=\int_{\Omega\cup\Omega_{\lambda}}
%\int_{\Omega_c}\frac{u(x)-u(y)}{|x-y|^{d+2\alpha}}(v(x)-v(y))dydx+\int_{\Omega_c}\int_{R^d}\frac{u(x)-u(y)}{|x-y|^{d+2\alpha}}(v(x)-v(y))dydx\nonumber\\
%&\qquad+\int_{\Omega\cup\Omega_{\lambda}}
%\int_{(\Omega\cup\Omega_{\lambda})\cap B_{\lambda}(x)}\frac{u(x)-u(y)}{|x-y|^{d+2\alpha}}(v(x)-v(y))dydx\nonumber\\
%&\qquad+\int_{\Omega\cup\Omega_{\lambda}}
%\int_{(\Omega\cup\Omega_{\lambda})\setminus B_{\lambda}(x)}\frac{u(x)-u(y)}{|x-y|^{d+2\alpha}}(v(x)-v(y))dydx\nonumber\\
%&\qquad-\int_{\Omega\cup\Omega_{\lambda}}
%\int_{(\Omega\cup\Omega_{\lambda})\cap B_{\lambda}(x)}\frac{{u}_\lambda(x)-{u}_\lambda(y)}{|x-y|^{d+2\alpha}}(v(x)-v(y))dydx\nonumber\\
%&=\int_{\Omega\cup\Omega_{\lambda}}
%\int_{(\Omega\cup\Omega_{\lambda})\cap B_{\lambda}(x)}\frac{e(x)-e(y)}{|x-y|^{d+2\alpha}}(v(x)-v(y))dydx\nonumber\\
%&\quad+\int_{\Omega\cup\Omega_{\lambda}}
%\int_{\Omega_c}\frac{u(x)-u(y)}{|x-y|^{d+2\alpha}}(v(x)-v(y))dydx+\int_{\Omega_c}\int_{R^d}\frac{u(x)-u(y)}{|x-y|^{d+2\alpha}}(v(x)-v(y))dydx\nonumber\\
%&\qquad+\int_{\Omega\cup\Omega_{\lambda}}
%\int_{(\Omega\cup\Omega_{\lambda})\setminus B_{\lambda}(x)}\frac{u(x)-u(y)}{|x-y|^{d+2\alpha}}(v(x)-v(y))dydx\nonumber
%\end{align}
Note that $u\in H_{\Omega}^{\alpha}(R^d)\cap H^1_{\Omega}(R^d)\subset H^1_{\Omega}(\Omega\cup\Omega_{\lambda})$ and $u$ is solenoidal, so $u \in V$ and therefore $e\in V$. Set $v=e\in V$, and rearrange the terms, then we have
\begin{equation}\label{trancate_eqn}
\left\{\begin{aligned}
& \frac{1}{2} \frac{d}{dt}\Vert e \Vert^2_{L^2(\Omega)} +\nu \Vert\nabla e \Vert^2_{L^2(\Omega)}+\gamma \frac{C_{d,\alpha}}{2}\vertiii {e}_{\alpha, \Omega\cup\Omega_{\lambda}} \\
&\quad=-{b_\Omega} (e, u, e)
-\gamma \frac{C_{d,\alpha}}{2}\int_{\Omega_c}\int_{R^d}\frac{u(x)-u(y)}{|x-y|^{d+2\alpha}}(e(x)-e(y))dydx\\
&\qquad-\gamma \frac{C_{d,\alpha}}{2}\int_{\Omega\cup\Omega_{\lambda}}
\int_{R^d\setminus B_{\lambda}(x)}\frac{u(x)-u(y)}{|x-y|^{d+2\alpha}}(e(x)-e(y))dydx
\end{aligned}\right.
\end{equation}
By Young's inequality, the nonlinear term can be bounded as (taking $\epsilon=\gamma^{3/4}$)
\begin{align}
{b_\Omega}(e,u,e)&\leq \Vert e \Vert_{L^2(\Omega)}^{\frac{1}{2}}\Vert \nabla e\Vert_{L^2(\Omega)}^{\frac{3}{2}}\Vert \nabla u\Vert_{L^2(\Omega)}\\
&\leq \frac{3}{4}\left(\epsilon\Vert \nabla e\Vert_{L^2(\Omega)}^{\frac{3}{2}}\right)^{\frac{4}{3}} + \frac{1}{4} \left( \frac{1}{\epsilon}\Vert e \Vert_{L^2(\Omega)}^{\frac{1}{2}}\Vert \nabla u \Vert_{L^2(\Omega)} \right)^4 \nonumber\\
&\leq \frac{3}{4}\epsilon^\frac{4}{3}\Vert \nabla e\Vert_{L^2(\Omega)}^2 + \frac{1}{4\epsilon^4}\Vert e\Vert_{L^2(\Omega)}^2 \Vert \nabla u \Vert_{L^2(\Omega)}^4\nonumber\\
&\leq \frac{3}{4}\gamma\Vert \nabla e\Vert_{L^2(\Omega)}^2 + \frac{1}{4 \gamma^3}\Vert e \Vert_{L^2(\Omega)}^2 \Vert \nabla u \Vert_{L^2(\Omega)}^4\nonumber
\end{align}
For all $y\in \Omega$, we have $B_{\lambda} (y) \subset \Omega \cup \Omega_{\lambda}$, and thus $\Omega_c =R^d \setminus (\Omega\cup\Omega_{\lambda}) \subset R^d \setminus B_{\lambda} (y) $. So
\begin{align}
&\gamma \frac{C_{d,\alpha}}{2}\left|\int_{\Omega_c}\int_{R^d}\frac{u(x)-u(y)}{|x-y|^{d+2\alpha}}(e(x)-e(y))dydx\right |\\
&=\gamma \frac{C_{d,\alpha}}{2} 
\left|\int_{\Omega_c}\int_{R^d}\frac{u(x)e(x)-u(y)e(x)-u(x)e(y)}{|x-y|^{d+2\alpha}}dydx+\int_{\Omega_c}
\int_{R^d}\frac{u(y)e(y)}{|x-y|^{d+2\alpha}}dydx\right |\nonumber\\
&=\gamma \frac{C_{d,\alpha}}{2} 
\left|\int_{\Omega_c}
\int_{R^d}\frac{u(y)e(y)}{|x-y|^{d+2\alpha}}dydx\right |\nonumber\\
&=\gamma \frac{C_{d,\alpha}}{2} 
\left|\int_{\Omega_c}
\int_{\Omega}\frac{u(y)e(y)}{|x-y|^{d+2\alpha}}dydx\right |\nonumber\\
&=\gamma \frac{C_{d,\alpha}}{2} 
\left|\int_{\Omega}u(y)e(y)
\int_{\Omega_c}\frac{1}{|x-y|^{d+2\alpha}}dxdy\right |\nonumber\\
&\leq \gamma \frac{C_{d,\alpha}}{2} 
\left|\int_{\Omega} u(y)e(y)
\int_{R^d \setminus B_{\lambda}(y)} \frac{1}{|x-y|^{d+2\alpha}}dxdy\right |\nonumber\\
&=\gamma \frac{C_{d,\alpha}}{2} 
\left|\int_{\Omega} u(y)e(y)
\int_{\lambda}^{+ \infty} \frac{1}{
|\rho|^{1+2\alpha}}d\rho dy\right |\nonumber\\
&=\gamma \frac{C_{d,\alpha}}{2} \frac{1}{2s\lambda^{2s}}
\left|\int_{\Omega} u(y)e(y) dy\right |\nonumber\\
&\leq\gamma \frac{C_{d,\alpha}}{2} \frac{1}{2s\lambda^{2s}} \Vert u\Vert_{L^2(\Omega)}\Vert e\Vert_{L^2(\Omega)}\nonumber\\
&\leq C\gamma \frac{C_{d,\alpha}}{2} \frac{1}{2s\lambda^{2s}} \Vert u\Vert_{L^2(\Omega)}\Vert \nabla e\Vert_{L^2(\Omega)}\nonumber\\
&\leq C\frac{\gamma^2}{\nu} \left(\frac{C_{d,\alpha}}{s\lambda^{2s}} \right)^2\Vert u\Vert_{L^2(\Omega)}^2+\frac{1}{32}\nu\Vert \nabla e\Vert_{L^2(\Omega)}^2\nonumber
\end{align}
\begin{align}
&\gamma \frac{C_{d,\alpha}}{2}\left|\int_{\Omega\cup\Omega_{\lambda}}
\int_{R^d \setminus B_{\lambda}(x) }\frac{u(x)-u(y)}{|x-y|^{d+2\alpha}}(e(x)-e(y))dydx\right |\\
&=\gamma \frac{C_{d,\alpha}}{2}\Bigg |\int_{\Omega}
\int_{R^d \setminus B_{\lambda}(x) }\frac{u(x)-u(y)}{|x-y|^{d+2\alpha}}(e(x)-e(y))dydx \nonumber\\
&\qquad+ \int_{\Omega_{\lambda}}
\int_{R^d \setminus B_{\lambda}(x) }\frac{u(x)-u(y)}{|x-y|^{d+2\alpha}}(e(x)-e(y))dydx\Bigg |\nonumber\\
&=\gamma \frac{C_{d,\alpha}}{2}\Bigg |\int_{\Omega}
\int_{R^d \setminus B_{\lambda}(x) }\frac{u(x)-u(y)}{|x-y|^{d+2\alpha}}(e(x)-e(y))dydx+ \int_{\Omega_{\lambda}}
\int_{\Omega \setminus B_{\lambda}(x) }\frac{u(y)e(y)}{|x-y|^{d+2\alpha}}dydx\Bigg |\nonumber\\
&\leq \gamma \frac{C_{d,\alpha}}{2}\Bigg |\int_{\Omega}
\int_{R^d \setminus B_{\lambda}(x) }\frac{u(x)-u(y)}{|x-y|^{d+2\alpha}}(e(x)-e(y))dydx\Bigg | \nonumber\\
&\qquad+ \gamma \frac{C_{d,\alpha}}{2}\Bigg | \int_{\Omega_{\lambda}}
\int_{\Omega \setminus B_{\lambda}(x) }\frac{u(y)e(y)}{|x-y|^{d+2\alpha}}dydx\Bigg |\nonumber
\end{align}
Let $w_d= 1, \pi, \frac{4}{3}\pi$, for $d=1,2,3$, respectively. 
\begin{align}
&\gamma \frac{C_{d,\alpha}}{2} \left |\int_{\Omega_{\lambda}}
\int_{\Omega \setminus B_{\lambda}(x) }\frac{u(y)e(y)}{|x-y|^{d+2\alpha}}dydx\right|\\
&\leq \gamma \frac{C_{d,\alpha}}{2} \int_{\Omega_{\lambda}}
\int_{\Omega \setminus B_{\lambda}(x) }\frac{|u(y)e(y)|}{|x-y|^{d+2\alpha}}dydx\nonumber\\
&\leq \gamma \frac{C_{d,\alpha}}{2} \int_{\Omega_{\lambda}}
\int_{\Omega \setminus B_{\lambda}(x) }\frac{|u(y)e(y)|}{\lambda^{d+2\alpha}}dydx\nonumber\\
&\leq \gamma \frac{C_{d,\alpha}}{2} \int_{\Omega_{\lambda}}
\int_{\Omega  }\frac{|u(y)e(y)|}{\lambda^{d+2\alpha}}dydx\nonumber\\
&=\gamma \frac{C_{d,\alpha}}{2}\frac{1}{\lambda^{d+2\alpha}} \int_{\Omega_{\lambda}}
\int_{\Omega  }|u(y)e(y)|dydx\\
&= \gamma \frac{C_{d,\alpha}}{2}\frac{Vol(\Omega_{\lambda})}{\lambda^{d+2\alpha}} 
\int_{\Omega  }|u(y)e(y)|dy\nonumber\\
&\leq  \gamma \frac{C_{d,\alpha}}{2}\frac{Vol(B_{\lambda+I})-Vol(\Omega)}{\lambda^{d+2\alpha}} 
\int_{\Omega  }|u(y)e(y)|dy\nonumber\\
&=  \gamma \frac{C_{d,\alpha}}{2}\frac{w_d(\lambda+I)^d-Vol(\Omega)}{\lambda^{d+2\alpha}} 
\int_{\Omega  }|u(y)e(y)|dy\nonumber\\
&\leq  \gamma \frac{C_{d,\alpha}}{2}\frac{w_d(\lambda+I)^d}{\lambda^{d+2\alpha}} 
\Vert u \Vert_{L^2(\Omega)}\Vert e\Vert_{L^2(\Omega)}\nonumber\\
&\leq C \gamma \frac{C_{d,\alpha}}{2}\frac{w_d(\lambda+I)^d}{\lambda^{d+2\alpha}} 
\Vert u \Vert_{L^2(\Omega)}\Vert \nabla e\Vert_{L^2(\Omega)}\nonumber\\
&\leq C \frac{\gamma^2}{\nu} C_{d,\alpha}^2\left(\frac{w_d(\lambda+I)^d}{\lambda^{d+2\alpha}} \right)^2
\Vert u \Vert_{L^2(\Omega)}^2+\frac{1}{32}\nu\Vert \nabla e\Vert_{L^2(\Omega)}^2\nonumber
\end{align}
\begin{align}
& \gamma \frac{C_{d,\alpha}}{2}\Bigg |\int_{\Omega}
\int_{R^d \setminus B_{\lambda}(x) }\frac{u(x)-u(y)}{|x-y|^{d+2\alpha}}(e(x)-e(y))dydx\Bigg | \nonumber\\
&\leq \underbrace{\gamma \frac{C_{d,\alpha}}{2}\int_{\Omega}
\int_{R^d \setminus B_{\lambda}(x) }\frac{|u(x)e(x)|}{|x-y|^{d+2\alpha}}dydx}_{\mbox{Term1}}
+\underbrace{ \gamma \frac{C_{d,\alpha}}{2}\int_{\Omega}
\int_{R^d \setminus B_{\lambda}(x) }\frac{|u(x)e(y)|}{|x-y|^{d+2\alpha}}dydx}_{\mbox{Term2}} \nonumber\\
& +\underbrace{\gamma \frac{C_{d,\alpha}}{2}\int_{\Omega}
\int_{R^d \setminus B_{\lambda}(x) }\frac{|u(y)e(x)|}{|x-y|^{d+2\alpha}}dydx}_{\mbox{Term3}}
+\underbrace{\gamma \frac{C_{d,\alpha}}{2}\int_{\Omega}
\int_{R^d \setminus B_{\lambda}(x) }\frac{|u(y)e(y)|}{|x-y|^{d+2\alpha}}dydx}_{\mbox{Term4}}\nonumber
\end{align}
Now we bound each of the terms above.
\begin{align}
 Term1 &= \gamma \frac{C_{d,\alpha}}{2}\int_{\Omega}
\int_{R^d \setminus B_{\lambda}(x) }\frac{|u(x)e(x)|}{|x-y|^{d+2\alpha}}dydx\\
&= \gamma \frac{C_{d,\alpha}}{2}\int_{\Omega} |u(x)e(x)|
\int_{R^d \setminus B_{\lambda}(x) }\frac{1}{|x-y|^{d+2\alpha}}dydx\nonumber\\
&= \gamma \frac{C_{d,\alpha}}{2}\int_{\Omega} |u(x)e(x)|
\int_{\lambda }^{+\infty}\frac{1}{|\rho |^{1+2\alpha}}d\rho dx\nonumber\\
&= \gamma \frac{C_{d,\alpha}}{2}\frac{1}{2\alpha\lambda^{2\alpha}}\int_{\Omega} |u(x)e(x)| dx\nonumber\\
&\leq \gamma \frac{C_{d,\alpha}}{2}\frac{1}{2\alpha\lambda^{2\alpha}}\Vert u\Vert_{L^2(\Omega)}\Vert e \Vert_{L^2(\Omega)}\nonumber\\
&\leq C\gamma \frac{C_{d,\alpha}}{2}\frac{1}{2\alpha\lambda^{2\alpha}}\Vert u\Vert_{L^2(\Omega)}\Vert \nabla e \Vert_{L^2(\Omega)}\nonumber\\
&\leq C\frac{\gamma^2}{\nu} \left(\frac{C_{d,\alpha}}{\alpha\lambda^{2\alpha}}\right)^2\Vert u\Vert^2_{L^2(\Omega)}+\frac{1}{32}\nu\Vert \nabla e \Vert_{L^2(\Omega)}\nonumber
\end{align}
\begin{align}
 Term 2 &= \gamma \frac{C_{d,\alpha}}{2}\int_{\Omega}
\int_{R^d \setminus B_{\lambda}(x) }\frac{|u(x)e(y)|}{|x-y|^{d+2\alpha}}dydx\\
&= \gamma \frac{C_{d,\alpha}}{2}\int_{\Omega} |u(x)|
\int_{R^d \setminus B_{\lambda}(x) }\frac{|e(y)|}{|x-y|^{d+2\alpha}}dydx\nonumber\\
&= \gamma \frac{C_{d,\alpha}}{2}\int_{\Omega} |u(x)|
\int_{\Omega \setminus B_{\lambda}(x) }\frac{|e(y)|}{|x-y|^{d+2\alpha}}dydx\nonumber\\
&\leq \gamma \frac{C_{d,\alpha}}{2}\int_{\Omega} |u(x)|
\int_{\Omega \setminus B_{\lambda}(x) }\frac{|e(y)|}{\lambda^{d+2\alpha}}dydx\nonumber\\
&= \gamma \frac{C_{d,\alpha}}{2}\frac{1}{\lambda^{d+2\alpha}}\int_{\Omega} |u(x)|
\int_{\Omega \setminus B_{\lambda}(x) }|e(y)|dydx\nonumber\\
&\leq \gamma \frac{C_{d,\alpha}}{2}\frac{1}{\lambda^{d+2\alpha}}\int_{\Omega} |u(x)|
\int_{\Omega }|e(y)|dydx\nonumber\\
&= \gamma \frac{C_{d,\alpha}}{2}\frac{1}{\lambda^{d+2\alpha}}\Vert u\Vert_{L^1(\Omega)}
\int_{\Omega }\left(|e(y)|\cdot 1\right) dy\nonumber\\
&\leq \gamma \frac{C_{d,\alpha}}{2}\frac{1}{\lambda^{d+2\alpha}}\sqrt{Vol(\Omega)}\Vert u\Vert_{L^1(\Omega)}\Vert e\Vert_{L^2(\Omega)}\nonumber\\
&\leq C\gamma \frac{C_{d,\alpha}}{2}\frac{1}{\lambda^{d+2\alpha}}\sqrt{Vol(\Omega)}\Vert u\Vert_{L^1(\Omega)}\Vert \nabla e\Vert_{L^2(\Omega)}\nonumber\\
&\leq C\frac{\gamma^2}{\nu} 
\left(\frac{C_{d,\alpha}}{\lambda^{d+2\alpha}}\right)^2 Vol(\Omega)\Vert u\Vert_{L^1(\Omega)}^2+ \frac{1}{32}\nu\Vert \nabla e\Vert_{L^2(\Omega)}^2\nonumber
\end{align}

\begin{align}
 Term 3 &= \gamma \frac{C_{d,\alpha}}{2}\int_{\Omega}
\int_{R^d \setminus B_{\lambda}(x) }\frac{|u(y)e(x)|}{|x-y|^{d+2\alpha}}dydx\\
&= \gamma \frac{C_{d,\alpha}}{2}\int_{\Omega} |e(x)|
\int_{R^d \setminus B_{\lambda}(x) }\frac{|u(y)|}{|x-y|^{d+2\alpha}}dydx\nonumber\\
&= \gamma \frac{C_{d,\alpha}}{2}\int_{\Omega} |e(x)|
\int_{\Omega \setminus B_{\lambda}(x) }\frac{|u(y)|}{|x-y|^{d+2\alpha}}dydx\nonumber\\
&\leq \gamma \frac{C_{d,\alpha}}{2}\int_{\Omega} |e(x)|
\int_{\Omega \setminus B_{\lambda}(x) }\frac{|u(y)|}{\lambda^{d+2\alpha}}dydx\nonumber\\
&= \gamma \frac{C_{d,\alpha}}{2}\frac{1}{\lambda^{d+2\alpha}}\int_{\Omega} |e(x)|
\int_{\Omega \setminus B_{\lambda}(x) }|u(y)|dydx\nonumber\\
&\leq \gamma \frac{C_{d,\alpha}}{2}\frac{1}{\lambda^{d+2\alpha}}\int_{\Omega} |e(x)|
\int_{\Omega }|u(y)|dydx\nonumber\\
&= \gamma \frac{C_{d,\alpha}}{2}\frac{1}{\lambda^{d+2\alpha}}\int_{\Omega} |e(x)|dx
\int_{\Omega }|u(y)|dy\nonumber\\
&= \gamma \frac{C_{d,\alpha}}{2}\frac{1}{\lambda^{d+2\alpha}}\Vert u\Vert_{L^1(\Omega)}
\int_{\Omega }\left(|e(x)|\cdot 1\right) dx\nonumber\\
&\leq \gamma \frac{C_{d,\alpha}}{2}\frac{1}{\lambda^{d+2\alpha}}\sqrt{Vol(\Omega)}\Vert u\Vert_{L^1(\Omega)}\Vert e\Vert_{L^2(\Omega)}\nonumber\\
&\leq C\gamma \frac{C_{d,\alpha}}{2}\frac{1}{\lambda^{d+2\alpha}}\sqrt{Vol(\Omega)}\Vert u\Vert_{L^1(\Omega)}\Vert \nabla e\Vert_{L^2(\Omega)}\nonumber\\
&\leq C\frac{\gamma^2}{\nu} 
\left(\frac{C_{d,\alpha}}{\lambda^{d+2\alpha}}\right)^2 Vol(\Omega)\Vert u\Vert_{L^1(\Omega)}^2+ \frac{1}{32}\nu\Vert \nabla e\Vert_{L^2(\Omega)}^2\nonumber
\end{align}
\begin{align}
 Term 4 &= \gamma \frac{C_{d,\alpha}}{2}\int_{\Omega}
\int_{R^d \setminus B_{\lambda}(x) }\frac{|u(y)e(y)|}{|x-y|^{d+2\alpha}}dydx\\
 &= \gamma \frac{C_{d,\alpha}}{2}\int_{\Omega}
\int_{\Omega\setminus B_{\lambda}(x) }\frac{|u(y)e(y)|}{|x-y|^{d+2\alpha}}dydx\nonumber\\
&\leq \gamma \frac{C_{d,\alpha}}{2}\int_{\Omega} 
\int_{\Omega \setminus B_{\lambda}(x) }\frac{|u(y)e(y)|}{\lambda^{d+2\alpha}}dydx\nonumber\\
&= \gamma \frac{C_{d,\alpha}}{2}\frac{1}{\lambda^{d+2\alpha}}\int_{\Omega} 
\int_{\Omega \setminus B_{\lambda}(x) }|u(y)e(y)|dydx\nonumber\\
&\leq \gamma \frac{C_{d,\alpha}}{2}\frac{1}{\lambda^{d+2\alpha}}\int_{\Omega} 
\int_{\Omega }|u(y)e(y)|dydx\nonumber\\
&= \gamma \frac{C_{d,\alpha}}{2}\frac{Vol(\Omega)}{\lambda^{d+2\alpha}}
\int_{\Omega }|u(y)e(y)|dy\nonumber\\
&\leq \gamma \frac{C_{d,\alpha}}{2}\frac{Vol(\Omega)}{\lambda^{d+2\alpha}}
\Vert u\Vert_{L^2(\Omega)}\Vert e\Vert_{L^2(\Omega)}\nonumber\\
&\leq C\gamma \frac{C_{d,\alpha}}{2}\frac{Vol(\Omega)}{\lambda^{d+2\alpha}}
\Vert u\Vert_{L^2(\Omega)}\Vert\nabla e\Vert_{L^2(\Omega)}\nonumber\\
&\leq C\frac{\gamma^2}{\nu} C_{d,\alpha}^2\left(\frac{Vol(\Omega)}{\lambda^{d+2\alpha}}\right)^2
\Vert u\Vert_{L^2(\Omega)}^2+\frac{1}{32}\nu\Vert\nabla e\Vert_{L^2(\Omega)}^2\nonumber
\end{align}
After bounding all the terms on the right hand side of \eqref{trancate_eqn}, we now have
\begin{align}
&\frac{1}{2} \frac{d}{dt}\Vert e \Vert^2_{L^2(\Omega)} +\frac{1}{16}\nu \Vert\nabla e \Vert^2_{L^2(\Omega)}+\gamma \frac{C_{d,\alpha}}{2}\vertiii {e}_{\alpha, \Omega\cup\Omega_{\lambda}} \label{trancate: final}\\
& \leq  \frac{1}{4 \gamma^3}\Vert \nabla u \Vert^4_{L^2(\Omega)}\Vert e \Vert^2_{L^2(\Omega)} +C\frac{\gamma^2}{\nu} \left(\frac{C_{d,\alpha}}{\alpha\lambda^{2\alpha}} \right)^2\Vert u\Vert_{L^2(\Omega)}^2+C \frac{\gamma^2}{\nu} C_{d,\alpha}^2\left(\frac{w_d(\lambda+I)^d}{\lambda^{d+2\alpha}} \right)^2
\Vert u \Vert_{L^2(\Omega)}^2\nonumber\\
 &+C\frac{\gamma^2}{\nu} 
\left(\frac{C_{d,\alpha}}{\lambda^{d+2\alpha}}\right)^2 Vol(\Omega)\Vert u\Vert_{L^1(\Omega)}^2+ C\frac{\gamma^2}{\nu} C_{d,\alpha}^2\left(\frac{Vol(\Omega)}{\lambda^{d+2\alpha}}\right)^2
\Vert u\Vert_{L^2(\Omega)}^2\nonumber
\end{align}
Let $\phi(t)=\frac{1}{2\nu^3}\int_{0}^t\Vert \nabla u(x,s)\Vert^4_{L^2(\Omega)}ds$. Multiplying \eqref{trancate: final} with $exp(-\phi(t))$ gives
\begin{align}
&\qquad\frac{d}{dt} \left( \exp(-\phi(t))\Vert e\Vert^2_{L^2(\Omega)} \right)+ exp(-\phi(t))\frac{1}{16}\nu \Vert\nabla e \Vert^2_{L^2(\Omega)}+exp(-\phi(t))\gamma \frac{C_{d,\alpha}}{2}\vertiii {e}_{\alpha, \Omega\cup\Omega_{\lambda}}
\label{trancate:1}\\
&\qquad \leq exp(-\phi(t))\cdot \Bigg (C\frac{\gamma^2}{\nu} \left(\frac{C_{d,\alpha}}{\alpha\lambda^{2\alpha}} \right)^2\Vert u\Vert_{L^2(\Omega)}^2
+C \frac{\gamma^2}{\nu} C_{d,\alpha}^2\left(\frac{w_d(\lambda+I)^d}{\lambda^{d+2\alpha}} \right)^2
\Vert u \Vert_{L^2(\Omega)}^2\nonumber\\
 &\qquad+C\frac{\gamma^2}{\nu} 
\left(\frac{C_{d,\alpha}}{\lambda^{d+2\alpha}}\right)^2 Vol(\Omega)\Vert u\Vert_{L^1(\Omega)}^2
+ C\frac{\gamma^2}{\nu} C_{d,\alpha}^2\left(\frac{Vol(\Omega)}{\lambda^{d+2\alpha}}\right)^2
\Vert u\Vert_{L^2(\Omega)}^2\Bigg )\nonumber
\end{align}
Note that $e(x,0)=0$. Integrating above inequality over $[0,t]$ and then multiply both sides with $exp(\phi(t))$ gives
\begin{align}
& \Vert e(x,t)\Vert^2_{L^2(\Omega)}+exp(\phi(t))
\int_{0}^t exp(-\phi(s))\frac{1}{16}\nu \Vert\nabla e(x,s) \Vert^2_{L^2(\Omega)} ds\label{trancate:2}\\
&\qquad\qquad+exp(\phi(t))\int_{0}^t exp(-\phi(s))\gamma \frac{C_{d,\alpha}}{2}\vertiii {e(x,s)}_{\alpha, \Omega\cup\Omega_{\lambda}} ds \nonumber\\
&\qquad \leq exp(\phi(t))\int_{0}^t exp(-\phi(s))\cdot \Bigg (C\frac{\gamma^2}{\nu} \left(\frac{C_{d,\alpha}}{\alpha\lambda^{2\alpha}} \right)^2\Vert u(x,s)\Vert_{L^2(\Omega)}^2\nonumber\\
&\qquad+C \frac{\gamma^2}{\nu} C_{d,\alpha}^2\left(\frac{w_d(\lambda+I)^d}{\lambda^{d+2\alpha}} \right)^2
\Vert u(x,s) \Vert_{L^2(\Omega)}^2\nonumber\\ &\qquad+C\frac{\gamma^2}{\nu} 
\left(\frac{C_{d,\alpha}}{\lambda^{d+2\alpha}}\right)^2 Vol(\Omega)\Vert u(x,s)\Vert_{L^1(\Omega)}^2
+ C\frac{\gamma^2}{\nu} C_{d,\alpha}^2\left(\frac{Vol(\Omega)}{\lambda^{d+2\alpha}}\right)^2
\Vert u(x,s)\Vert_{L^2(\Omega)}^2\Bigg ) ds\nonumber\\
\end{align}
Note that $exp(\phi(T))\geq exp(\phi(s))\geq exp(\phi(0))$ for $\forall\, s\in [0,T]$. So we have 
\begin{align}
& \Vert e(x,t)\Vert^2_{L^2(\Omega)}+\int_{0}^t \frac{1}{16}\nu \Vert\nabla e(x,s) \Vert^2_{L^2(\Omega)} ds+\int_{0}^t \gamma \frac{C_{d,\alpha}}{2}\vertiii {e(x,s)}_{\alpha, \Omega\cup\Omega_{\lambda}} ds\label{trancate:3}\\
&\qquad \leq exp(\phi(T)-\phi(0))\Bigg \{\left(C\frac{\gamma^2 C_{d,\alpha}^2}{\nu \alpha^2}\int_{0}^T  \Vert u(x,s)\Vert_{L^2(\Omega)}^2 ds\right)\cdot \left(\frac{1}{\lambda^{2\alpha}}\right) \nonumber\\
&\qquad \quad +\left(C \frac{\gamma^2 C_{d,\alpha}^2w_d^2}{\nu} \int_{0}^T \Vert u(x,s) \Vert_{L^2(\Omega)}^2 ds\right) \left(\frac{(\lambda+I)^d}{\lambda^{d+2\alpha}} \right)^2
\nonumber\\
 &\qquad\quad+\left( C\frac{\gamma^2 C_{d,\alpha}^2 Vol(\Omega)}{\nu}  \int_{0}^T \Vert u(x,s)\Vert_{L^1(\Omega)}^2 ds \right)
\left(\frac{1}{\lambda^{d+2\alpha}}\right)^2 \nonumber\\
&\qquad\quad+\left( C\frac{\gamma^2 C_{d,\alpha}^2 Vol(\Omega)^2}{\nu} \int_{0}^T \Vert u(x,s)\Vert_{L^2(\Omega)}^2 ds\right)\cdot\left(\frac{1}{\lambda^{d+2\alpha}}\right)^2\Bigg \}
\nonumber
\end{align}
Because $u\in L^2(0,T; L^2(\Omega))$, after absorbing constants, we have
\begin{align}
& \Vert e(x,t)\Vert^2_{L^2(\Omega)}+\int_{0}^t \frac{1}{16}\nu \Vert\nabla e(x,s) \Vert^2_{L^2(\Omega)} ds+\int_{0}^t \gamma \frac{C_{d,\alpha}}{2}\vertiii {e(x,s)}_{\alpha, \Omega\cup\Omega_{\lambda}} ds\label{trancate:4}\\
&\qquad\leq  C\left[ \left(\frac{1}{\lambda^{2\alpha}}\right)^2 + \left(\frac{(\lambda+I)^d}{\lambda^{d+2\alpha}} \right)^2+
\left(\frac{1}{\lambda^{d+2\alpha}}\right)^2 \right]
\nonumber
\end{align}
In particular, if $\lambda\geq 1$, we have
\begin{align}
\left(\frac{\lambda+I}{\lambda}\right)^d=\left(1+\frac{I}{\lambda}\right)^d\leq (1+I)^d \quad\text{ and } \quad\frac{1}{\lambda^d}\leq 1,
\end{align}
and thus, after absorbing constants, we have
\begin{align}
&\Vert e(x,T)\Vert^2_{L^2(\Omega)}+\int_{0}^T \frac{1}{16}\nu \Vert\nabla e(x,s) \Vert^2_{L^2(\Omega)} ds\label{trancate:5}\\
&\qquad\qquad\qquad\qquad+\int_{0}^T \gamma \frac{C_{d,\alpha}}{2}\vertiii {e(x,s)}_{\alpha, \Omega\cup\Omega_{\lambda}} ds
\leq  C\left(\frac{1}{\lambda^{2\alpha}}\right) ^2.\nonumber
\end{align}

}

%%%%%%%%%%%%

\subsection{Proof of Theorem \ref{fullydiscreteerror}}\label{thm5.1proof}

{\allowdisplaybreaks

To facilitate error analysis, we consider the weakly divergence free velocity space.
$$
V_h:=\{ v\in X_h: (\nabla\cdot v, q)=0, \forall q\in Q_h\}.
$$
Then, the solution of the truncated variational problem \eqref{prob} ${u_\lambda}$ satisfies
\begin{equation}\label{prob-erroranalysis}
\left\{\begin{aligned}
& \int_{\Omega }\frac{{u_\lambda^{n+1}}-{u_\lambda^{n}}}{\Delta t}v dx +{\widetilde b}_{\Omega} ({u_\lambda^{n+1}}, {u_\lambda^{n+1}}, v)+\nu \int_{\Omega} \nabla {u_\lambda^{n+1}} : \nabla v dx\\
& \qquad+\gamma \frac{C_{d,\alpha}}{2}\int_{\Omega \cup \Omega_\lambda }\int_{(\Omega\cup\Omega_\lambda)\cap\mathcal{B}_{\lambda}(x)} \frac{{u_\lambda^{n+1}}(x)-{u_\lambda^{n+1}}(y)}{|x-y|^{n+2\alpha}}(v(x)-v(y))dydx\\
& \qquad-\int_{\Omega} {p_\lambda^{n+1}} (\nabla\cdot v)dx= \int_{\Omega} f^{n+1} v dx +\int_{\Omega} R({u_\lambda^{n+1}}) v dx \qquad\forall v\in V_h,
\end{aligned}\right.
\end{equation}
where $R({u_\lambda^{n+1}})$ is defined as
\[
R({u_\lambda^{n+1}})=\frac{{u_\lambda^{n+1}}-{u_\lambda^{n}}}{\Delta t}-{u_\lambda}_{t}
(t^{n+1})\text{ .}
\]
Assuming that $X_h$ and $Q_h$ satisfy the discrete LBB condition, the method \eqref{full} is equivalent to: given $u_{h,\lambda}^n$, find $u_{h,\lambda}^{n+1}$ satisfying
\begin{equation}\label{full-Vh}
\left\{\begin{aligned}
&\int_{\Omega}\frac{u_{h,\lambda}^{n+1}-u_{h,\lambda}^n}{\Delta t} v dx +{\widetilde b}_{\Omega}(u_{h,\lambda}^n,u_{h,\lambda}^{n+1},v)+\nu\int_\Omega \nabla u_{h,\lambda}^{n+1}:\nabla v dx 
\\
&+\gamma\frac{C_{d,\alpha}}{2}\int_{\Omega\cup
\Omega_{\lambda}}\int_{\Omega\cup
\Omega_\lambda\cap
\mathcal{B}_\lambda(x)}\frac{u_{h,\lambda}^{n+1}(x)-u_{h,\lambda}^{n+1}(x^\prime)}{|x-x^\prime|^{d+2\alpha}}(v(x)-v(x^\prime))dx^\prime dx
\\
&\qquad\qquad\qquad\qquad\qquad\qquad\qquad\qquad\qquad=\int_{\Omega} f^{n+1} v dx\qquad\forall v\in V_h.
\end{aligned}\right.
\end{equation}
With $e^{n}_{h,\lambda}={u_\lambda}^n-u_{h,\lambda}^n $, subtracting \eqref{full-Vh} from  (\ref{prob-erroranalysis}) gives
\begin{equation}\label{eq:err}
\begin{aligned}
&\int_{\Omega} \frac{e^{n+1}_{h,\lambda}-e^{n}_{h,\lambda}}{ \Delta t} v dx
+{\widetilde b}_{\Omega}({u_\lambda^{n+1}},{u_\lambda^{n+1}},v)-{\widetilde b}_{\Omega}(u_{h,\lambda}^n,u_{h,\lambda}^{n+1},v)\\
&+\gamma\frac{C_{d,\alpha}}{2}\int_{\Omega\cup
\Omega_{\lambda}}\int_{\Omega\cup
\Omega_\lambda\cap
\mathcal{B}_\lambda(x)}\frac{e^{n+1}_{h,\lambda}(x)-e^{n+1}_{h,\lambda}(x^\prime)}{|x-x^\prime|^{d+2\alpha}}(v(x)-v(x^\prime))dx^\prime dx
\\
&+\nu\int_\Omega \nabla e^{n+1}_{h,\lambda}:\nabla v dx 
-\int_{\Omega} {p_\lambda^{n+1}} (\nabla\cdot v)dx= \int_{\Omega} R({u_\lambda^{n+1}}) v dx\text{ .}
\end{aligned}
\end{equation}
We split $e^{n}_{h,\lambda}$ into two terms
\[
e^{n}_{h,\lambda}={u_\lambda}^n-u_{h,\lambda}^n = ({u_\lambda}^n-I_h {u_\lambda}^n)+ (I_h{u_\lambda}^n-u_{h,\lambda}^n)=\eta^n + \xi_{h}^n,
\]
where $I_h{u_\lambda}^n \in V_h$ is the interpolant of ${u_\lambda}^n$ in $V_h$. Now setting $v=\xi_h^{n+1}\in V_h$ in \eqref{eq:err} and multiplying through by $\Delta t$ gives
\begin{equation}\label{eq:err1}
\begin{aligned}
&\frac{1}{2}\Vert \xi_h^{n+1}\Vert_{L^2(\Omega)}^{2}-\frac{1}{2}\Vert \xi_h^{n}\Vert_{L^2(\Omega)}^{2}+\frac{1}{2}\|\xi_h^{n+1}-\xi_h^{n}\|^{2}_{L^2(\Omega)}\\
&+\Delta t\nu
\Vert \nabla \xi_h^{n+1}\Vert^{2}_{L^2(\Omega)}
+\Delta t\gamma \frac{C_{d,\alpha}}{2} \vertiii{\xi_h^{n+1}}_{\alpha, \Omega\cup\Omega_{\lambda}}^2\\
&=-\Delta t {\widetilde b}_{\Omega}({u_\lambda^{n+1}},{u_\lambda^{n+1}},\xi_h^{n+1})
+\Delta t {\widetilde b}_{\Omega}(u_{h,\lambda}^n,u_{h,\lambda}^{n+1},\xi_h^{n+1})\\
&\quad+\Delta t \int_{\Omega}R({u_\lambda^{n+1}})\cdot \xi_h
^{n+1}dx+\Delta t\int_{\Omega} {p_\lambda^{n+1}} (\nabla\cdot \xi_h^{n+1})dx\\
&-\int_{\Omega} (\eta^{n+1}-\eta^{n})\cdot \xi_h^{n+1} dx-\Delta t\nu\int_\Omega \nabla \eta^{n+1}:\nabla \xi_h^{n+1} dx\\
&+\Delta t\gamma\frac{C_{d,\alpha}}{2}\int_{\Omega\cup
\Omega_{\lambda}}\int_{\Omega\cup
\Omega_\lambda\cap
\mathcal{B}_\lambda(x)}\frac{\eta^{n+1}(x)-\eta^{n+1}(x^\prime)}{|x-x^\prime|^{d+2\alpha}}(\xi_h^{n+1}(x)-\xi_h^{n+1}(x^\prime))dx^\prime dx
\text{ .}
\end{aligned}
\end{equation}
Next, we bound the nonlinear terms. Subtracting and adding  ${\widetilde b}_{\Omega}(u_{h,\lambda}^{ n},{u_\lambda^{n+1}},\xi_h^{n+1})$ and ${\widetilde b}_{\Omega}(u_{\lambda}^{ n},{u_\lambda^{n+1}},\xi_h^{n+1})$ and using skew symmetry, we have
\begin{equation}\label{eq:nonlinear222}
\begin{aligned}
{\widetilde b}_{\Omega}(u_{h,\lambda}^{n},u_{h,\lambda}^{n+1}&,\xi_h^{n+1})
-{\widetilde b}_{\Omega}({u_\lambda^{n+1}}, {u_\lambda^{n+1}},\xi_h^{n+1})\\
&= {\widetilde b}_{\Omega}(u_{h,\lambda}^{n},u_{h,\lambda}^{n+1},\xi_h^{n+1})-{\widetilde b}_{\Omega}(u_{h,\lambda}^{n}
,{u_\lambda^{n+1}},\xi_h^{n+1})\\
&\quad+{\widetilde b}_{\Omega}(u_{h,\lambda}^{n}
,u_{\lambda}^{n+1},\xi_h^{n+1})-{\widetilde b}_{\Omega}({u_\lambda}^{ n},u_{\lambda}^{n+1},\xi_h^{n+1})\\
&\quad+{\widetilde b}_{\Omega}(u_{\lambda}^{ n},u_{\lambda}^{n+1},\xi_h^{n+1})-{\widetilde b}_{\Omega}(u_{\lambda}^{ n+1},{u_\lambda^{n+1}},\xi_h^{n+1})\\
&=- {\widetilde b}_{\Omega}(u_{h,\lambda}^{n},e^{n+1}_{h,\lambda},\xi_h^{n+1})-{\widetilde b}_{\Omega}(e^{n}_{h,\lambda}
,{u_\lambda^{n+1}},\xi_h^{n+1})\\
&\quad-{\widetilde b}_{\Omega}({u_\lambda^{n+1}}-u_{\lambda}^n, {u_\lambda^{n+1}}, \xi_h^{n+1})\\
&=- {\widetilde b}_{\Omega}(u_{h,\lambda}^{n},\eta^{n+1},\xi_h^{n+1})-{\widetilde b}_{\Omega}(\eta^{n}
,{u_\lambda^{n+1}},\xi_h^{n+1})\\
&\quad-{\widetilde b}_{\Omega}(\xi_h^{n}
,{u_\lambda^{n+1}},\xi_h^{n+1})-{\widetilde b}_{\Omega}({u_\lambda^{n+1}}-u_{\lambda}^n, {u_\lambda^{n+1}}, \xi_h^{n+1}).
\end{aligned}
\end{equation}
With the assumption ${u_\lambda}\in L^{\infty}(0,T; H^1(\Omega))$, we estimate the nonlinear terms as follows:
\begin{align}
\Delta t \vert {\widetilde b}_{\Omega}(\eta^{n},{u_\lambda^{n+1}},\xi_{h}^{n+1})\vert &\leq C\Delta t\|\nabla \eta
^{n}\|_{L^2(\Omega)}\|\nabla {u_\lambda^{n+1}}\|_{L^2(\Omega)}\|\nabla\xi_{h}^{n+1}\|_{L^2(\Omega)}\\
&\leq\Delta t\frac{\nu}{64}
\|\nabla\xi_{h}^{n+1}\|_{L^2(\Omega)}^{2}+C\Delta t\nu^{-1}\|\nabla\eta^{n+1}\|_{L^2(\Omega)}^{2}\text{ ,}\nonumber
\end{align}
\begin{align}
\Delta t\vert {\widetilde b}_{\Omega}({u_\lambda^{n+1}}&-{u_\lambda^{n}},{u_\lambda^{n+1}},\xi_{h}^{n+1}) \vert\nonumber\\
&\leq C\Delta t\|\nabla
({u_\lambda^{n+1}}-{u_\lambda^{n}})\|_{L^2(\Omega)}\|\nabla {u_\lambda^{n+1}}\|_{L^2(\Omega)}\|\nabla\xi_{h}
^{n+1}\|_{L^2(\Omega)}\nonumber\\
&\leq\Delta t\frac{\nu}{64}
\|\nabla\xi_{h}^{n+1}\|_{L^2(\Omega)}^{2}+C\Delta t\nu^{-1}\|\nabla({u_\lambda}
^{n+1}-{u_\lambda^{n}})\|_{L^2(\Omega)}^{2}\nonumber\\
&\leq\Delta t\frac{\nu}{64}
\|\nabla\xi_{h}^{n+1}\|_{L^2(\Omega)}^{2}+\frac{C{\Delta t}^{3}}{\nu
}\|\frac{\nabla ({u_\lambda^{n+1}}-{u_\lambda^{n}})}{\Delta t}\|_{L^2(\Omega)}^{2}\nonumber\\
&= \Delta t\frac{\nu}{64}
\|\nabla\xi_{h}^{n+1}\|_{L^2(\Omega)}^{2}+\frac{C{\Delta t}^{3}}{\nu} (
\int_{\Omega}(\frac{1}{\Delta t}\int_{t^{n}}^{t^{n+1}} (\nabla {u_\lambda}_{t}
)dt)^{2} d\Omega)\\
&\leq\Delta t\frac{\nu}{64}
\|\nabla\xi_{h}^{n+1}\|_{L^2(\Omega)}^{2}+\frac{C{\Delta t}^{3}}{\nu}
(\int_{\Omega}(\frac{1}{\Delta t}\int_{t^{n}}^{t^{n+1}}|\nabla {u_\lambda}_{t}|^{2}
dt) d\Omega)\nonumber\\
&\leq\Delta t\frac{\nu}{64}
\|\nabla\xi_{h}^{n+1}\|_{L^2(\Omega)}^{2}+\frac{C\Delta t^2}{\nu}
(\int_{t^{n}}^{t^{n+1}}\| \nabla {u_\lambda}_{t} \|_{L^2(\Omega)}^{2} dt)\text{ , }\nonumber
\end{align}
\begin{equation}\label{eq:nonlinear2221}
\begin{aligned}
\Delta t\vert {\widetilde b}_{\Omega}(\xi_h^{n},& {u_\lambda^{n+1}}, \xi_h^{n+1}) \vert\\
&\leq C\Delta t\| \nabla \xi_h
^{n}\|^{\frac{1}{2}}_{L^2(\Omega)} \Vert \xi_h^{n}\Vert^{\frac{1}{2}}_{L^2(\Omega)}\|\nabla
{u_\lambda^{n+1}}\|_{L^2(\Omega)}\|\nabla \xi_h^{n+1}\|_{L^2(\Omega)}\\
&\leq C\|\nabla
{u_\lambda^{n+1}}\|_{L^2(\Omega)}\left(\epsilon\Delta t\|\nabla \xi_h^{n+1}\|^{2}_{L^2(\Omega)}+\frac{1}{\epsilon}\Delta t\|\nabla \xi_h
^{n} \|_{L^2(\Omega)}\|\xi_h^{n} \|_{L^2(\Omega)}\right)\\
&\leq C\left(\epsilon\Delta t\|\nabla \xi_h^{n+1}\|^{2}_{L^2(\Omega)}+\frac{1}{\epsilon}\left(\delta \Delta t
\|\nabla \xi_h^n \|^{2}_{L^2(\Omega)}+\frac{1}{\delta}\Delta t\|\xi_h^{n}\|_{L^2(\Omega)}^2\right)\right)\\
&\leq \left(\frac{\nu}{8} \Delta t\|\nabla \xi_h^{n+1}\|^{2}_{L^2(\Omega)}+\frac{\nu}{8} \Delta t\|\nabla
\xi_h^{n}\|^{2}_{L^2(\Omega)}\right)+\frac{C}{\nu^{3}}\Delta t\|\xi_h^{n} \|^{2}_{L^2(\Omega)}.
\end{aligned}
\end{equation}
Using Young's inequality and the result from the stability analysis, i.e., $\Vert u_{h,\lambda}^n \Vert^2 \leq C$, we have
\begin{align}
\Delta t\vert {\widetilde b}_{\Omega}(u^{n}_{h,\lambda}, &\eta^{n+1}, \xi_{h}^{n+1})\vert \nonumber\\
&\leq C\Delta t\|\nabla u^{n}_{h,\lambda}
\|_{L^2(\Omega)}^{1/2}\Vert u_{h,\lambda}^n\Vert_{L^2(\Omega)}^{1/2}\|\nabla \eta^{n+1}\|_{L^2(\Omega)}\|\nabla\xi_{h}^{n+1}\|_{L^2(\Omega)}\\
&\leq\Delta t\frac{\nu}{64}
\|\nabla\xi_{h}^{n+1}\|_{L^2(\Omega)}^{2}+C\Delta t\nu^{-1}\|\nabla u_{h,\lambda}^{n+1}\|_{L^2(\Omega)}\|\nabla\eta^{n}
\|_{L^2(\Omega)}^{2} \nonumber
\end{align}
Next, consider the pressure term. Since $\xi_{h}^{n+1}\in V_{h}$
we have
\begin{align}
\Delta t({p_\lambda^{n+1}},\nabla\cdot\xi_{h}^{n+1})&=\Delta t({p_\lambda^{n+1}}-q_{h}^{n+1},
\nabla\cdot\xi_{h}^{n+1})\nonumber\\
&\leq\Delta t\|{p_\lambda^{n+1}}-q_{h}^{n+1}\|_{L^2(\Omega)}\|\nabla\cdot\xi_{h}^{n+1}\|_{L^2(\Omega)}\\
&\leq\Delta t\frac{\nu}{64}\|\nabla\xi_{h}^{n+1}\|_{L^2(\Omega)}^{2}+C\Delta t \nu^{-1}\|{p_\lambda}
^{n+1}-q_{h}^{n+1}\|_{L^2(\Omega)}^{2} \text{ .}\nonumber
\end{align}
Finally,
\begin{equation}\label{lastineq}
\begin{aligned}
\Delta t
\int_{\Omega} R({u_\lambda^{n+1}})&\cdot \xi_h^{n+1} dx\\
&=\Delta t\int_{\Omega}\left(\frac{{u_\lambda^{n+1}}-{u_\lambda^{n}}}{\Delta
t}-{u_\lambda}_{t}(t^{n+1})\right) \cdot \xi_h^{n+1}dx\\
&\leq C\Delta t\|\frac{{u_\lambda^{n+1}}-{u_\lambda^{n}}}{\Delta t}-{u_\lambda}_{t}(t^{n+1})\| _{L^2(\Omega)}\|\nabla
\xi_h^{n+1}\|_{L^2(\Omega)}\\
&\leq\frac{\nu}{16}\Delta t\|\nabla \xi_h^{n+1}\|^{2}_{L^2(\Omega)}+\frac{C}{\nu}\Delta t\|\frac
{{u_\lambda^{n+1}}-{u_\lambda^{n}}}{\Delta t}-{u_\lambda}_{t}(t^{n+1})\|^{2}_{L^2(\Omega)}\\
&\leq\frac{\nu}{16}\Delta t\|\nabla \xi_h^{n+1}\|^{2}_{L^2(\Omega)}+\frac{C\Delta t^2}{\nu}
\int_{t^{n}}^{t^{n+1}}\|{u_\lambda}_{tt}\|^{2}_{L^2(\Omega)} dt \text{ .}
\end{aligned}
\end{equation}
\begin{align}
-\int_{\Omega}(\eta^{n+1}-\eta^{n})\cdot \xi_{h}^{n+1}dx &\leq
C\Delta t\|\frac{\eta^{n+1}-\eta^{n}}{ \Delta t}\|_{L^2(\Omega)} \|\nabla\xi_{h}
^{n+1}\|_{L^2(\Omega)}\nonumber\\
&\leq C \Delta t\bar{\nu}^{-1}\|\frac{\eta^{n+1}-\eta^{n}}{ \Delta t}\|_{L^2(\Omega)}^{2}+\Delta t\frac
{\nu}{64}\|\nabla\xi_{h}^{n+1}\|_{L^2(\Omega)}^{2}\\
&\leq C\Delta t \nu^{-1}\|\frac{1}{\Delta t}\int_{t^{n}}^{t^{n+1}} \eta_{t} \text{ }
dt \|_{L^2(\Omega)}^2+\Delta t\frac{\nu}{64}\|\nabla\xi_{h}^{n+1}\|_{\Omega}^{2}\nonumber\\
&\leq\frac{C}{\nu}\int_{t^{n}}^{t^{n+1}}\| \eta_{t}\|_{L^2(\Omega)}^{2}\text{ }
dt+\Delta t\frac{\nu}{64}
\|\nabla\xi_{h}^{n+1}\|_{L^2(\Omega)}^{2}\text{ ,}\nonumber
\end{align}
\begin{align}
-\Delta t\nu\int_{\Omega}\nabla\eta^{n+1} : \nabla\xi_{h}^{n+1} dx&\leq\Delta t\nu\|\nabla\eta
^{n+1}\|_{L^2(\Omega)} 
\|\nabla\xi_{h}^{n+1}\|_{L^2(\Omega)}\\
&\leq C\Delta t\nu\|\nabla\eta^{n+1}\|_{L^2(\Omega)}^{2}+\Delta t \frac{\nu}{64}\|\nabla\xi_{h}
^{n+1}\|_{L^2(\Omega)}^{2} \text{ ,}\nonumber
\end{align}
\begin{align}
-\Delta t\gamma\frac{C_{d,\alpha}}{2}\int_{\Omega\cup
\Omega_{\lambda}}\int_{\Omega\cup
\Omega_\lambda\cap
\mathcal{B}_\lambda(x)}&\frac{\eta^{n+1}(x)-\eta^{n+1}(x^\prime)}{|x-x^\prime|^{d+2\alpha}}(\xi_h^{n+1}(x)-\xi_h^{n+1}(x^\prime))dx^\prime dx\\
&\leq \Delta t\gamma\frac{C_{d,\alpha}}{2} \vertiii{\eta^{n+1}}_{\alpha, \Omega\cup\Omega_{\lambda}}\vertiii{\xi_{h}^{n+1}}_{\alpha, \Omega\cup\Omega_{\lambda}}\nonumber\\
&\leq \Delta t\gamma\frac{C_{d,\alpha}}{4} \vertiii{\eta^{n+1}}_{\alpha, \Omega\cup\Omega_{\lambda}}^2+\Delta t\gamma\frac{C_{d,\alpha}}{4}\vertiii{\xi_{h}^{n+1}}_{\alpha, \Omega\cup\Omega_{\lambda}}^2\nonumber\\
&\leq C \Delta t\gamma\frac{C_{d,\alpha}}{4} \Vert \eta^{n+1}\Vert_{H^{\alpha}(\Omega\cup\Omega_{\lambda})}^2
+\Delta t\gamma\frac{C_{d,\alpha}}{4}\vertiii{\xi_{h}^{n+1}}_{\alpha, \Omega\cup\Omega_{\lambda}}^2\nonumber\\
&\leq C \Delta t\gamma\frac{C_{d,\alpha}}{4} \Vert \eta^{n+1}\Vert_{H^{1}(\Omega\cup\Omega_{\lambda})}^2
+\Delta t\gamma\frac{C_{d,\alpha}}{4}\vertiii{\xi_{h}^{n+1}}_{\alpha, \Omega\cup\Omega_{\lambda}}^2\nonumber\\
&\leq C \Delta t\gamma\frac{C_{d,\alpha}}{4} \Vert \eta^{n+1}\Vert_{H^{1}(\Omega\cup\Omega_{\lambda})}^2
+\Delta t\gamma\frac{C_{d,\alpha}}{4}\vertiii{\xi_{h}^{n+1}}_{\alpha, \Omega\cup\Omega_{\lambda}}^2\nonumber\\
&\leq C \Delta t\gamma\frac{C_{d,\alpha}}{4} \Vert \eta^{n+1}\Vert_{H^{1}(\Omega)}^2
+\Delta t\gamma\frac{C_{d,\alpha}}{4}\vertiii{\xi_{h}^{n+1}}_{\alpha, \Omega\cup\Omega_{\lambda}}^2\nonumber
\end{align}
\noindent Combining, we now have the inequality
\begin{equation}
\begin{aligned}
&\frac{1}{2}\Vert \xi_h^{n+1}\Vert_{L^2(\Omega)}^{2}-\frac{1}{2}\Vert \xi_h^{n}\Vert_{L^2(\Omega)}^{2}+\frac{1}{2}\|\xi_h^{n+1}-\xi_h^{n}\|^{2}_{L^2(\Omega)}+\frac{1}{2}\Delta t\nu
\Vert \nabla \xi_h^{n+1}\Vert^{2}_{L^2(\Omega)}\\
&
+\frac{\nu}{8} \Delta t\left(\|\nabla
\xi_h^{n+1}\|^{2}_{L^2(\Omega)}-\|\nabla
\xi_h^{n}\|^{2}_{L^2(\Omega)}\right)+\Delta t\gamma \frac{C_{d,\alpha}}{4} \vertiii{\xi_h^{n+1}}_{\alpha, \Omega\cup\Omega_{\lambda}}^2\\
&\leq \frac{C}{\nu^{3}}\Delta t\|\xi_h^{n} \|^{2}_{L^2(\Omega)}
+C\Delta t\nu^{-1}
\|\nabla\eta^{n+1}\|_{L^2(\Omega)}^{2}
+\frac{C\Delta t^2}{\nu}
(\int_{t^{n}}^{t^{n+1}}\| \nabla {u_\lambda}_{t} \|_{L^2(\Omega)}^{2} dt)\\
&+C\Delta t\nu^{-1}\|\nabla u_{h,\lambda}^{n+1}\|_{L^2(\Omega)}\|\nabla\eta^{n}
\|_{L^2(\Omega)}^{2} 
+C\Delta t \nu^{-1}\|{p_\lambda}
^{n+1}-q_{h}^{n+1}\|_{L^2(\Omega)}^{2}\\
&+\frac{C\Delta t^2}{\nu}
\int_{t^{n}}^{t^{n+1}}\|{u_\lambda}_{tt}\|^{2}_{L^2(\Omega)} dt+C\nu^{-1}\int_{t^{n}}^{t^{n+1}}\| \eta_{t}\|_{L^2(\Omega)}^{2}\text{ }
dt\\
&+C\Delta t\nu\|\nabla\eta^{n+1}\|_{L^2(\Omega)}^{2}+C \Delta t\gamma\frac{C_{d,\alpha}}{4} \Vert \eta^{n+1}
\Vert_{H^1(\Omega)}^2
\text{ .}
\end{aligned}
\end{equation}
Taking the sum from $n=0$ to
$n=N-1$ gives
\begin{equation}
\begin{aligned}
&\frac{1}{2}\Vert \xi_h^{N}\Vert_{L^2(\Omega)}^{2}+\frac{\nu}{8} \Delta t\|\nabla
\xi_h^{N}\|^{2}_{L^2(\Omega)}+\frac{1}{2}\Delta t \sum_{n=0}^{N-1}\|\xi_h^{n+1}-\xi_h^{n}\|^{2}_{L^2(\Omega)}\\
&\qquad+\frac{1}{2}\Delta t\sum_{n=0}^{N-1}\nu
\Vert \nabla \xi_h^{n+1}\Vert^{2}_{L^2(\Omega)}
+\Delta t\gamma \frac{C_{d,\alpha}}{4} \vertiii{\xi_h^{n+1}}_{\alpha, \Omega\cup\Omega_{\lambda}}^2\\
&\leq \frac{C}{\nu^{3}}\Delta t\|\xi_h^{n} \|^{2}_{L^2(\Omega)}+\frac{1}{2}\Vert \xi_h^{0}\Vert_{L^2(\Omega)}^{2}+\frac{\nu}{8} \Delta t\|\nabla
\xi_h^{0}\|^{2}_{L^2(\Omega)}
+\sum_{n=0}^{N-1}\Big \{C\Delta t\nu^{-1}
\|\nabla\eta^{n+1}\|_{L^2(\Omega)}^{2}\\
&\quad+\frac{C\Delta t^2}{\nu}
(\int_{t^{n}}^{t^{n+1}}\| \nabla {u_\lambda}_{t} \|_{L^2(\Omega)}^{2} dt)+C\Delta t\nu^{-1}\|\nabla u_{h,\lambda}^{n+1}\|_{L^2(\Omega)}\|\nabla\eta^{n}
\|_{L^2(\Omega)}^{2} \\
&\quad+C\Delta t \nu^{-1}\|{p_\lambda}
^{n+1}-q_{h}^{n+1}\|_{L^2(\Omega)}^{2}
+\frac{C\Delta t^2}{\nu}
\int_{t^{n}}^{t^{n+1}}\|{u_\lambda}_{tt}\|^{2}_{L^2(\Omega)} dt\\
&\quad+C\nu^{-1}\int_{t^{n}}^{t^{n+1}}\| \eta_{t}\|_{L^2(\Omega)}^{2}\text{ }
dt
+C\Delta t\nu\|\nabla\eta^{n+1}\|_{L^2(\Omega)}^{2}+C \Delta t\gamma\frac{C_{d,\alpha}}{4} \Vert \eta^{n+1}
\Vert_{H^1(\Omega)}^2\Big \}
\text{ .}
\end{aligned}
\end{equation}
Because ${u_\lambda}\in H^{2}(0,T;L^{2}(\Omega))$, we have
\begin{align}
C\nu^{-1}\Delta t&\sum_{n=0}^{N-1}
\Vert\nabla u_{h.\lambda}^{n+1}\Vert_{L^2(\Omega)}
\Vert\nabla\eta^{n}\Vert_{L^2(\Omega)}^{2}
\\
&\leq C\nu^{-1}h^{2k}\Delta t \sum_{n=0}^{N-1}\Vert {u_\lambda}^n\Vert^2_{H^{k+1}(\Omega)}\Vert^{2}\Vert\nabla u_{h,\lambda}^{n+1}\Vert_{L^2(\Omega)}\nonumber\\
&\leq C\nu^{-1}h^{2k}\left( \Delta t \sum_{n=0}^{N-1}\Vert {u_\lambda}^n\Vert_{H^{k+1}(\Omega)}^4+\Delta t \sum_{n=0}^{N-1}\Vert\nabla u_{h,\lambda}^{n+1}\Vert^2 \right)\nonumber\\
&\leq C\nu^{-1}h^{2k}\Vert\vert {u_\lambda}\vert\Vert^4_{4, k+1, \Omega}+C\nu^{-1}h^{2k}\nonumber
\end{align}
Using the interpolation inequality and the stability result, i.e., $\Delta t \sum_{n=0}^{N-1}\Vert u_{h,\lambda}^n\Vert_{L^2(\Omega)}^2\leq C$, we have
\begin{equation}
\begin{aligned}
&\frac{1}{2}\Vert \xi_h^{N}\Vert_{L^2(\Omega)}^{2}+\frac{\nu}{8} \Delta t\|\nabla
\xi_h^{N}\|^{2}_{L^2(\Omega)}+\frac{1}{2}\Delta t \sum_{n=0}^{N-1}\|\xi_h^{n+1}-\xi_h^{n}\|^{2}_{L^2(\Omega)}\\
&\qquad+\frac{1}{2}\Delta t\sum_{n=0}^{N-1}\nu
\Vert \nabla \xi_h^{n+1}\Vert^{2}_{L^2(\Omega)}
+\Delta t\gamma \frac{C_{d,\alpha}}{4} \vertiii{\xi_h^{n+1}}_{\alpha, \Omega\cup\Omega_{\lambda}}^2\\
&\leq \frac{C}{\nu^{3}}\Delta t\|\xi_h^{n} \|^{2}_{L^2(\Omega)}+\frac{1}{2}\Vert \xi_h^{0}\Vert_{L^2(\Omega)}^{2}+\frac{\nu}{8} \Delta t\|\nabla
\xi_h^{0}\|^{2}_{L^2(\Omega)}
+C\nu^{-1}h^{2k}
\|\vert {u_\lambda}\vert\|_{2,k+1,\Omega}^{2}\\
&\quad+\frac{C\Delta t^2}{\nu}
\| \nabla {u_\lambda}_{t} \|_{L^2(0,T;L^2(\Omega))}^{2} +C\nu^{-1}h^{2k}\Vert\vert {u_\lambda}\vert\Vert^4_{4, k+1, \Omega}+C\nu^{-1}h^{2k}\\
&\quad+C\nu^{-1}h^{2s+2}\|\vert{p_\lambda}
\vert\|_{2,s+1,\Omega}^{2}
+\frac{C\Delta t^2}{\nu}
\|{u_\lambda}_{tt}\|^{2}_{L^2(0,T;L^2(\Omega))} \\
&\quad+C\nu^{-1}h^{2k+2}\| {u_\lambda}_{t}\|_{L^2(0,T;H^{k+1}(\Omega))}^{2}
+C\nu h^{2k}
\|\vert {u_\lambda}\vert\|_{2,k+1,\Omega}^{2}+C\gamma h^{2k}\Vert \vert {u_\lambda}\vert\Vert_{2,k+1,\Omega }^2
\text{ .}
\end{aligned}
\end{equation}
The next step results from the application of the discrete Gronwall
inequality \cite[p. 176]{GR79}:
\begin{equation}
\begin{aligned}
&\frac{1}{2}\Vert \xi_h^{N}\Vert_{L^2(\Omega)}^{2}+\frac{\nu}{8} \Delta t\|\nabla
\xi_h^{N}\|^{2}_{L^2(\Omega)}+\frac{1}{2}\Delta t \sum_{n=0}^{N-1}\|\xi_h^{n+1}-\xi_h^{n}\|^{2}_{L^2(\Omega)}\\
&\qquad+\frac{1}{2}\Delta t\sum_{n=0}^{N-1}\nu
\Vert \nabla \xi_h^{n+1}\Vert^{2}_{L^2(\Omega)}
+\Delta t\gamma \frac{C_{d,\alpha}}{4} \vertiii{\xi_h^{n+1}}_{\alpha, \Omega\cup\Omega_{\lambda}}^2\\
&\leq exp(\frac{CN\Delta t}{\nu^3})\Bigg\{\frac{\nu}{8} \Delta t\|\nabla
\xi_h^{0}\|^{2}_{L^2(\Omega)}+\frac{1}{2}\Vert \xi_h^{0}\Vert_{L^2(\Omega)}^{2}
+C\nu^{-1}h^{2k}
\|\vert {u_\lambda}\vert\|_{2,k+1,\Omega}^{2}\\
&\quad+\frac{C\Delta t^2}{\nu}
\| \nabla {u_\lambda}_{t} \|_{L^2(0,T;L^2(\Omega))}^{2} +C\nu^{-1}h^{2k}\Vert\vert {u_\lambda}\vert\Vert^4_{4, k+1, \Omega}+C\nu^{-1}h^{2k}\\
&\quad+C\nu^{-1}h^{2s+2}\|\vert{p_\lambda}
\vert\|_{2,s+1,\Omega}^{2}
+\frac{C\Delta t^2}{\nu}
\|{u_\lambda}_{tt}\|^{2}_{L^2(0,T;L^2(\Omega))} \\
&\quad+C\nu^{-1}h^{2k+2}\| {u_\lambda}_{t}\|_{L^2(0,T;H^{k+1}(\Omega))}^{2}
+C\nu h^{2k}
\|\vert {u_\lambda}\vert\|_{2,k+1,\Omega}^{2}+C\gamma h^{2k}\Vert \vert {u_\lambda}\vert\Vert_{2,k+1,\Omega }^2\Bigg \}
\text{ .}
\end{aligned}
\end{equation}
Recall that $e^{n}_{h,\lambda}=\eta^{n}+\xi_{h}^{n}$. Using the triangle
inequality on the error equation to split the error terms into the terms 
$\eta^{n}$ and $\xi_{h}^{n}$ gives
\begin{align}
&\frac{1}{2}\Vert e_{h,\lambda}^{N}\Vert_{L^2(\Omega)}^{2}+\Delta t\sum_{n=0}^{N-1}\frac{\nu}{2}\Vert\nabla e^{n+1}_{h,\lambda}\Vert_{L^2(\Omega)}
^{2}\\
&\qquad\leq\frac{1}{2}\Vert\xi_{h}^{N}\Vert_{L^2(\Omega)}^{2}
+\Delta t\sum_{n=0}^{N-1}\frac{\nu}{2}\Vert\nabla \xi_{h}^{n+1}\Vert_{L^2(\Omega)}
^{2}+\frac{1}{2}\Vert\eta^{N}\Vert_{L^2(\Omega)}^{2}
+\Delta t\sum_{n=0}^{N-1}\frac{\nu}{2}\Vert\nabla \eta^{n+1}\Vert_{L^2(\Omega)}
^{2}\text{ ,}\nonumber
\end{align}
and 
\begin{align}
\frac{1}{2}\Vert\xi_{h}^{0}&\Vert_{L^2(\Omega)}^{2}
+\frac{\nu\Delta t}{8}\Vert\nabla \xi_{h}
^{0}\Vert_{L^2(\Omega)}^{2}\nonumber\\
&\leq \frac{1}{2}\Vert e^{0}_{h,\lambda}\Vert_{L^2(\Omega)}^{2}+\frac{\nu\Delta t}{8}\Vert\nabla e
^{0}_{h,\lambda}\Vert_{L^2(\Omega)}^{2}+\frac{1}{2}\Vert\eta^{0}\Vert_{L^2(\Omega)}^{2}
+\frac{\nu\Delta t}{8}\Vert\nabla \eta
^{0}\Vert_{L^2(\Omega)}^{2}.
\end{align}
Using the previous bounds
for the $\eta^{n}$ terms and absorbing constants into a new constant $C$, we
obtain \eqref{full-err}.
}

\clearpage

\end{document}